\pdfoutput=1

\documentclass[a4paper,11pt]{article}
\usepackage[margin=1in]{geometry}
\usepackage[T1]{fontenc}
\usepackage[utf8]{inputenc}
\usepackage{lmodern}
\usepackage[english,greek,frenchb]{babel}
\usepackage{url,csquotes}
\usepackage[hidelinks,hyperfootnotes=false]{hyperref}
\usepackage{amsfonts,amssymb,enumerate}
\usepackage{amsmath}
\usepackage{graphicx}
\usepackage{bbm}
\usepackage{enumitem}
\usepackage{mdwlist}
\usepackage{listings}
\usepackage{dsfont}
\usepackage{amsthm}
\usepackage{url}
\usepackage{stmaryrd}
\usepackage[dvipsnames]{xcolor}
\usepackage{graphicx,import}
\usepackage{booktabs}
\usepackage{hyperref}
\hypersetup{colorlinks=true}

\DeclareMathOperator{\e}{e}
\newcommand{\1}{\mathbbm{1}}
\newcommand{\E}{\mathbb{E}}
\newcommand{\Prb}{\mathbb{P}}
\newcommand{\R}{\mathbb{R}}

\newcommand{\Sph}{\mathbb{S}}
\newcommand{\T}{\mathbb{T}}
\newcommand{\Z}{\mathbb{Z}}
\newcommand{\calA}{\mathcal{A}}

\newcommand{\calC}{\mathcal{C}}
\newcommand{\calD}{\mathcal{D}}
\newcommand{\calE}{\mathcal{E}}
\newcommand{\calF}{\mathcal{F}}

\newcommand{\calL}{\mathcal{L}}
\newcommand{\calM}{\mathcal{M}}
\newcommand{\calN}{\mathcal{N}}

\newcommand{\calS}{\mathcal{S}}
\newcommand{\calW}{\mathcal{W}}

\newtheorem{Thm}{Theorem}
\newtheorem{Prop}{Proposition}
\newtheorem{Lem}{Lemma}

\newtheorem{Def}{Definition}

\title{Vlasov limit for a chain of oscillators with Kac potentials}
\author{Alejandro Fernandez Montero
\thanks{CMAP, Ecole Polytechnique, CNRS, 91128 Palaiseau Cedex-France; E-mail: 
   \href{mailto:alejandro.fernandez-montero@polytechnique.edu}{alejandro.fernandez-montero@polytechnique.edu}
   }
}
\date{}

\begin{document}
\selectlanguage{english}
\maketitle




\begin{abstract}
We consider a chain of anharmonic oscillators with local mean field interaction and long-range stochastic exchanges of velocity. Even if the particles are not exchangeable, we prove the convergence of the empirical measure associated with this chain to a solution of a Vlasov-type equation. We then use this convergence to prove energy diffusion for a restricted class of anharmonic potentials.\\
\end{abstract}


\section{Model and results}

\subsection{Introduction}

The study of chains of interacting oscillators has drawn a lot of attention over the past few years. Deriving Fourier's law from an anharmonic chain is a major open problem in statistical mechanics \cite{Lebowitz}. Mathematically, the model consists in a system of $N$ oscillators, whose displacement and momentum are denoted by $X^i\in\R^d$ and $V^i\in\R^d$ respectively for $1\leq i \leq N$. Particles interact via a Hamiltonian dynamics, with Hamiltonian given by
$$ \mathcal{H} = \sum_{i=1}^N \left(\frac{1}{2} |V^i|^2 + \frac{1}{2} \sum_k\phi_{k}W(X^i - X^{i+k}) + U(X^i)\right), $$
where generally $(\phi_{k})_{k\in\Z}$ is such that $\phi_k=0$ when $|k|> K$ for some integer value $K$, \textit{i.e.} the interaction is only between oscillators with close lattice index. $W$ is a pair potential modelling the interaction between particles and $U$ is a pinning potential. It is known since \cite{Rieder} that for nearest neighbor harmonic interaction (\textit{i.e} $K=1$, $W(x) = x^2$, $U(x) = x^2$) the transport of energy is ballistic and therefore Fourier's law is not valid. The study of the anharmonic chain seems nevertheless out of reach for the moment. However, the model has drawn attention over the past few years with on the one hand the proof of the convergence to the unique invariant measure for anharmonic chains coupled to two heat baths with different temperatures (see \textit{e.g.} \cite{Eckmann} and \cite{Carmona}), and on the other hand the study of energy transport in harmonic chains with additional conservative stochastic collisions that enable to derive hydrodynamic limits (see \cite{Basile_review} for a review). In fact, stochastic collisions give enough ergodicity to derive such limits, but for long time scales, calculations rely heavily on the harmonic structure of the interactions.

In this paper, we consider a chain with so-called Kac potentials (see \cite{Presutti} for a detailed introduction), \textit{i.e.} we define the coefficients $\phi_k$ by
\begin{equation}\label{definition phi} \phi_k = \frac{1}{\ell N} \phi\left(\frac{k}{\ell N}\right),\end{equation}
for $|k|\leq \ell N$, where $\ell$ is a small parameter and $\phi$ is a smooth even function, with support included in $[-1/2,1/2]$ and normalized so that $\int_{-1/2}^{1/2}\phi(r)dr =1$. Therefore, the model has a local mean field structure at macroscopic distance $\ell$. In addition to the Hamiltonian dynamics, we also add stochastic exchanges of velocity between neighbors at distance of order $\ell N$. To do so, we introduce a smooth function $\gamma$ with the same properties as $\phi$, modulating the intensity of the stochastic exchanges and define for $|k|\leq \ell N$
\begin{equation}\label{definition gamma} \gamma_k = \frac{1}{\ell} \int_{\left[\frac{k-1/2}{N}, \frac{k+1/2}{N}\right]} \gamma\left(\frac{r}{\ell}\right) dr.\end{equation}
We exchange the velocities of two neighbors at distance $k$ at rate $\bar\gamma \gamma_k$, where $\bar\gamma$ is a positive parameter that gives the global rate at which a particle undergoes an exchange of velocity. The stochastic exchanges conserve the total energy of the system. We use the local mean field structure of the problem to prove the convergence of the empirical measure associated with the particle system to a Vlasov-type equation, and prove diffusion of the energy for a class of anharmonic pinning potentials.

We also mention that another model of chain of oscillators with long-range interaction has also been studied in \cite{TS} and \cite{Suda}. In this model, the stochastic collisions are short-range and there is no local mean field structure in the mechanical interactions. The techniques used are then different from ours and are similar to the short-range case \cite{BBO09,JKO15}.

\subsection{Model and notations}

In our setting, particles are indexed by the discrete periodic lattice $\Z/N\Z$. For every $1\leq i \leq N$, we set $r^i=i/N$. More generally, in what follows, the letter $r$ will refer to a position in the periodic domain $\T=\R/\Z$ and $z=(x,v)$ stands for the phase space coordinates of a single particle in the set $E := \R^d\times \R^d$. The dynamics followed by $(X^i_t,V^i_t)$, for $1\leq i\leq N$, is
\begin{equation}\label{original system solid new formalism}
\left\{
\begin{array}{l}
  dX^i_t = V^i_t dt \\
  dV^i_t = -\left(\int_{\T\times E} \Phi_\ell(r^i-r')\nabla W(X^i_t-x')d\mu^N_t(r',z') + \nabla U(X^i_t)\right)dt \\
      \hspace{9cm}+ \int_{\R^d} (v' - V^i_{t^-})d\calN^{\mu^N,r^i}(t,v').
\end{array}
\right.
\end{equation}
In the system \eqref{original system solid new formalism}, we wrote the Hamiltonian contribution by introducing the empirical measure $\mu^N_t$ associated with the system of particles:
\begin{equation}\label{definition empirical measure} \mu^N_t = \frac{1}{N}\sum_{i=1}^N\delta_{r^i,X^i_t,V^i_t}.\end{equation}
$\calN^{\mu^N,r^i}$ is a point process on $\R_+\times \R^d$ that directly selects the new velocity of $V^i$ at rate $\bar \gamma$ among its neighbors' velocity. More precisely, it is given by
$$\calN^{\mu^N,r^i} = \sum_{n}\delta_{T^{i,n},V^{i,n}},$$
where $(T^{i,n})_{n\geq 0}$ is the set of jump times of a Poisson process with intensity $\bar \gamma$ and, for any $n\geq 0$, $V^{i,n}$ is a random variable whose law given $T^{i,n}$ is
$$\Prb\left(V^{i,n} = V^{i+k}_{T^{i,n}_-} \mid T^{i,n}\right) = \gamma_k,$$
for all $-\ell N \leq k \leq \ell N$.  In \eqref{original system solid new formalism}, we also used the notation
\begin{equation}\label{definition phi ell gamme ell} \Phi_\ell(u) = \frac{1}{\ell}\phi\left(\frac{u}{\ell}\right),\end{equation}
and we define similarly
$$\Gamma_\ell(u) = \frac{1}{\ell}\gamma \left(\frac{u}{\ell}\right).$$

\begin{center}
\begin{figure}[h]
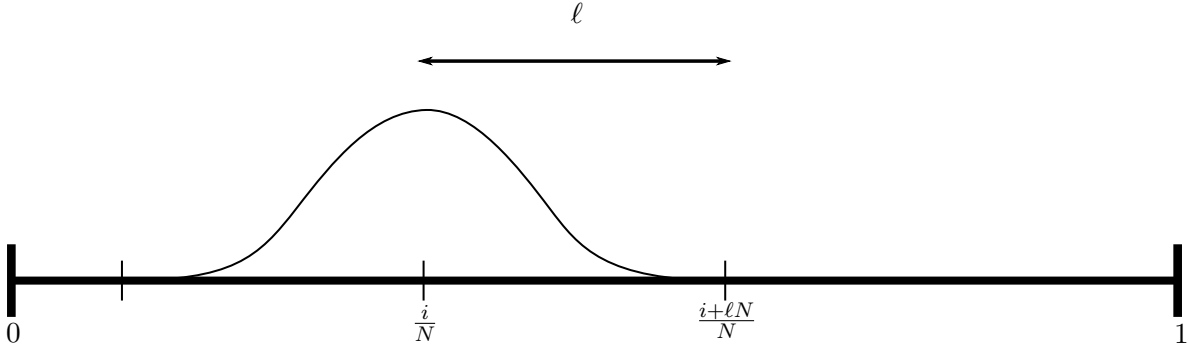
\caption{Different scales of the model and shape of the functions $\Phi_\ell$ and $\Gamma_\ell$}
\end{figure}
\end{center}

To properly define the dynamics for the whole system, we should specify the correlations between the point processes $\calN^{\mu^N,r^i}$, which is essentially that the jumps between two different pairs are independent. The detailed definition of those measures is left to Section \ref{Section 3}. The equation \eqref{original system solid new formalism} is thus a mean field type equation and we expect that the empirical measure $\mu^N_t$ converges to a solution of the following Vlasov-type equation
\begin{align}\label{Vlasov + noise}
&\partial_t f_t + v\cdot\nabla_xf_t - \Bigg(\int_{\T\times E} \Phi_\ell(r-r') \nabla W(x-x')f_t(r',z')dr'dz' + \nabla U(x)\Bigg)\cdot\nabla_vf_t \nonumber \\
                      &\hspace{3.5cm}=\bar\gamma \int_{\T\times E}\Gamma_\ell(r-r') \left(f_t(r,x,v')f_t(r',x',v) - f_t(r,z)f_t(r',z')\right)dr'dz'.
\end{align}
We will actually prove this statement by introducing the following nonlinear martingale problem. For any measure $\nu\in\calM^1(\T \times E)$, the space of probability measures on $\T \times E$, define the operator $\mathcal L[\nu]$ by
\begin{equation}\label{definition operator L} \mathcal L[\nu] \psi =  \mathcal A[\nu] \psi + \bar\gamma \mathcal S[\nu]\psi,\end{equation}
for all $\psi\in C_b^1(\T \times E)$, the space of bounded continuously differentiable real functions of $\T \times E$. $\mathcal A[\nu]$ is a drift operator given by
$$\mathcal A[\nu] \psi(r,x,v) = v \cdot \nabla_x\psi(r,x,v) - \left(\int_{\T \times E}\Phi_\ell(r-r')\nabla W(x-x') d\nu(r',z') + \nabla U(x)\right)\cdot \nabla_v\psi(r,x,v),$$
and
$$\mathcal S[\nu] \psi(r,x,v) = \int_{\T \times E} (\psi(r,x,v') - \psi(r,x,v)) \Gamma_\ell(r-r') d\nu(r',z')$$
is a pure jump operator that exchanges velocities. Denote by $\mathcal{D} = D(\R_+,\T\times E)$ the set of right continuous functions with left limits on $\R_+$ with values in $\T\times E$, by $\calF$ the product $\sigma$-field on $\calD$, and let $(\mathcal{F}_t)_{t\geq 0}$ be the filtration generated by the canonical process $Y=(r,Z)=(r,X,V)$ on $(\mathcal{D},\calF,\mu)$. Then the probability measure $\mu\in M^1(\mathcal D)$ is said to solve the \textit{nonlinear martingale problem} starting at $\nu_0\in \calM^1(\T\times E)$ if $\mu_0 = \nu_0$ and, for any $\psi\in C^1_b(\T\times E)$,
\begin{equation}\label{nonlinear mean field martingale problem}
M^{\psi}_t = \psi(Y_t) - \psi(Y_0) - \int_0^t \mathcal{L}[\mu_s] \psi(Y_s)ds
\end{equation}
is a martingale under $\mu$, $\mu_s$ denoting the time marginal of $\mu$ at time $s$. It is then straightforward to check that a solution to the martingale problem is a weak solution to the Vlasov-type equation \eqref{Vlasov + noise}. In particular, $r_t=r_0$ for any $t>0$, and we will indifferently use either the notation $r_t$, $r_0$ or $r$ to refer to the spatial coordinate of $(Y_t)_{t\geq 0}$ in $\T$. The martingale problem \eqref{nonlinear mean field martingale problem} is nonlinear in the sense that it is defined using the marginals $(\mu_t)_{t\geq 0}$, which are themselves only defined as a byproduct of the solutions to this problem. Taking expectations in \eqref{nonlinear mean field martingale problem}, it is then straightforward to check that the flow of time-marginals $(\mu_t)_{t\geq 0}$ associated with a solution $\mu$ to the martingale problem is a weak solution to the Vlasov-type equation \eqref{Vlasov + noise} in the sense of distributions.

Several existence and uniqueness results exist for the solutions of SDEs and nonlinear martingale problems with similar Poisson-type exchanges in the case of the Kac model (see for instance \cite{Desvillettes, FournierMeleard}, and also \cite{Graham} for a model with additional mean field interaction). Due to the non-exchangeability of the particles' coordinates in our model those results cannot be applied straightforwardly in our setting and we have to study it from scratch.

\subsection{Results}
We will always assume that the following hypotheses hold
\begin{enumerate}[label=\textbf{(H\arabic*)}]
\item \label{H1} $W, U \in \calC^2(\R^d,\R)$, the set of twice differentiable functions of $\R^d$ taking values in $\R$. Moreover $\nabla W$ and $\nabla U$ are uniformly Lipschitz and satisfy $\nabla W(0) = \nabla U(0)=0$. There exists a positive constant $c$ such that for all $x\in\R^d$,
$$|x|^2 \leq c U(x), \hspace{3cm} |\nabla W(x) |^2 \leq c W(x).$$
\item \label{H2} $\phi, \gamma \in \calC_c^\infty(\R)$, the set of infinitely differentiable functions of $\R$ with compact support, and in particular, their support is included in $[-1/2,1/2]$. Moreover $\phi$ and $\gamma$ are nonnegative, even functions that are non-increasing on $[0,1/2]$ and such that $\int_{-1/2}^{1/2}\phi(r)dr  = \int_{-1/2}^{1/2}\gamma(r)dr  = 1$.
\item \label{H7} $\int_{\T\times E}\left(\frac{1}{2}|v|^2 + \frac{1}{2}\int_{\T \times E} \Phi_\ell(r-r')W(x-x')d\mu_0(r',z') + U(x)\right)d\mu_0(r,z) < \infty.$
\item \label{H5} $\mu_0\in\calM^1(\T\times E)$ has a density $f_0$ with respect to the Lebesgue measure and its $r$-marginal is the uniform measure on $\T$. In particular, for any $r\in\T$, $f_0(r,\cdot)$ is a probability density on $E$.\\
Moreover, there exist a probability density $h$ on $E$ with finite first moment $\int_E |z| h(z) dz < \infty$, and a constant $C>0$ such that for any $r,r'\in\T$ and $z\in E$,
$$|f_0(r,z) - f_0(r',z)| \leq C |r-r'| h(z).$$
\end{enumerate}

The Lipschitz assumption in \ref{H1} is classical in mean field theory. The two inequalities on $U$ and $W$ are technical assumptions that hold for harmonic potentials. Hypothesis \ref{H7} is a moment assumption which, by conservation of energy, is crucial. \ref{H5} is the minimal regularity hypothesis on the $r$ variable at time $0$ that we will need to prove the mean field limit. It holds for local Gibbs measures, for which the temperature is a regular function of the $r$ variable for instance. Under these hypotheses, we then have the following proposition.
\begin{Prop}\label{nonlinear mean field martingale problem proposition}
There is a unique solution to the nonlinear martingale problem \eqref{nonlinear mean field martingale problem} starting at $\mu_0$.
\end{Prop}

In particular, Proposition \ref{nonlinear mean field martingale problem proposition} implies existence of weak solutions to \eqref{Vlasov + noise}. Denote by $\calW_1$ the Wasserstein distance associated with the Euclidean norm $| \cdot |$ on $\T \times E$
\begin{equation}\label{definition Wasserstein}\calW_{1} (\mu,\nu) = \inf_{\pi\in\Pi(\mu,\nu)} \int_{\T \times E} | y-y' |d\pi(y,y'),\end{equation}
where $\Pi(\mu,\nu)$ is the set of couplings of the probability measures $\mu,\nu\in\calM^1(\T \times E)$, and we used the shortened notation $y:=(r,z) =(r,x,v)$. We are now able to state the mean field convergence result.

\begin{Thm}\label{Convergence result}
Let $\mu$ be the solution of the nonlinear martingale problem \eqref{nonlinear mean field martingale problem} starting at $\mu_0$. Assume that the initial coordinates $(Z_0^i)_{i\leq N}$, with $Z_0^i = (X_0^i,V_0^i)$, are independent and with respective density distribution $f_0(i/N,z)$ for all $1\leq i \leq N$. Then there exist two positive constants $K_1$ and $K_2$ such that for any $1/N<\epsilon_N<\ell$,
$$\E\left[\mathcal{W}_1(\mu^N_t,\mu_t)\right] \leq K_1 \left((N\epsilon_N)^{-\frac{1}{4(d+1)}} + \frac{\epsilon_N}{\ell} + \frac{\bar\gamma}{1+\bar\gamma} \frac{\epsilon_N^{1/2}}{\ell^{1/2}}\right)\e^{K_2(1+\bar\gamma) t}.$$
\end{Thm}

The constants $K_1$ and $K_2$ in the theorem do not depend on the three parameters $N$, $\ell$ and $\bar\gamma$, but only on the potentials $W$ and $U$, on the functions $\phi$ and $\gamma$, and the initial measure $\mu_0$. $\epsilon_N$ is a coarse-graining parameter that naturally appears in the proof of Theorem \ref{Convergence result} and is precisely defined in Section \ref{Section 3}. In particular, for fixed $\bar\gamma$, choosing $\epsilon_N = \ell^{\frac{2d+2}{2d+3}} N^{-\frac{1}{2d+3}}$, we deduce that there exist positive constants $K$ and $K'$ such that for any $\ell > 1/N$, 
$$\E\left[\mathcal{W}_1(\mu^N_t,\mu_t)\right] \leq K (\ell N)^{-\frac{1}{2(2d+3)}} \e^{K' t}.$$

One of the main features of the particle system we consider is that, contrary to classical mean field theory for gases, the sequence $(Z^i_t)_{i\leq N}$ is not exchangeable. The behavior of $Z^i_t$ is intrinsically tied with the position $r^i$ in the chain. In particular, one cannot prove the mean field limit by comparing the law of one typical oscillator $Z^i_t$ at time $t$ to $\mu_t$, as usually done in the mean field theory (see \cite{Dobrushin, Golse, Sznitman}). Instead, the whole system $(r^i,Z^i_t)_{i\leq N}$ has to be compared to $\mu_t$ in its entirety.

Another difficulty comes from the fact that, even if $\nabla W$ is uniformly Lipschitz over $\R^d$, the map $(r,x)\mapsto \Phi_\ell(r)\nabla W(x)$ is not uniformly Lipschitz over $\T\times \R^d$ in general. Consequently, classical mean field limit proofs do not readily work in this situation, and even the proof of Proposition \ref{nonlinear mean field martingale problem proposition} is not straightforward. We bypass this difficulty for Proposition \ref{nonlinear mean field martingale problem proposition} by proving a contraction estimate for a well-suited distance, the sliced Wasserstein distance (see Definition \ref{sliced wasserstein}). Its proof is postponed to Section \ref{Section 2}.

Theorem \ref{Convergence result} is proved in Section \ref{Section 3}. The proof is based on a coupling of the particle system $(Y^i)_{i\leq N} := (r^i,Z^i)_{i\leq N}$ to a new system $(\widetilde Y^i)_{i\leq N}$, whose law is based on the solution $\mu$ to the nonlinear martingale problem \eqref{nonlinear mean field martingale problem}, and which are driven by the same Poisson measures as the original particle system. We use an original coupling over mesoscopic boxes in $\T$, and control directly the averages $1/N\sum\E[|Y^i-\widetilde Y^i|]$ to circumvent the aforementioned difficulties arising in this mean field limit. Moreover, contrary to classical McKean-Vlasov theory (see \cite{Sznitman} for instance), since the stochastic terms contribute via conservative exchanges, the new system $(\widetilde r^i, \widetilde Z^i)_{i\leq N}$ is heavily correlated. We will therefore introduce a third system of independent processes, based on the techniques developed by Cortez and Fontbona in \cite{Cortez} for the Kac model (see also \cite{Fournier2} for an example with continuous exchanges).

Finally, in Section \ref{Section 4}, we study transport of energy in appropriate scales for the system by using this mean field limit. We will consider the following additional moment and symmetry hypothesis on the initial distribution $f_0$:
\begin{enumerate}[label=\textbf{(H\arabic*)},resume]
\item \label{H8} for any $(r,z)\in\T\times E$, $f_0(r,z)=f_0(r,-z)$.
\item \label{H9} $\int_{\T\times E}\left( |v|^{2+2b} + |x|^{2+2b} \right) d\mu_0(r,x,v) <\infty$ for some $b>0$.
\end{enumerate}

We prove in Lemma \ref{Symmetry} that the symmetry \ref{H8} is preserved at any later time $t$ for the solution of the nonlinear martingale problem. Therefore, for a harmonic interaction potential $W(x)=|x|^2/2$, $\mu$ is a weak solution to the simpler equation
\begin{align}\label{simple Vlasov}
&\partial_t f_t + v\cdot\nabla_xf_t - \Bigg(x +  \nabla U(x)\Bigg)\cdot\nabla_vf_t \nonumber \\
                      &\hspace{3.5cm}= \bar\gamma \int_{\T\times E}\Gamma_\ell(r-r') \left(f_t(r,x,v')f_t(r',x',v) - f_t(r,z)f_t(r',z')\right)dr'dz',
\end{align}
and the term coming from the interaction potential is therefore reduced to an additional pinning term (associated with a harmonic potential). Therefore, energy is only transmitted by the noise in the mean field limit. Define the energy of particle $i$ at time $t$ by
\begin{equation}\label{definition energy particle system}\calE^i_t := \frac{1}{2} |V_t^i|^2 + U(X_t^i) +\frac{1}{4}\sum_{k=-\ell N}^{\ell N} \phi_k | X_t^i - X_t^{i+k} |^2,\end{equation}
we prove that $\calE^i_t$ evolves diffusively for a class of anharmonic pining potentials $U$.

\begin{Prop}\label{Energy evolution}
Suppose $W$ is harmonic and $U(x) = |x|^2\psi(x/|x|)$, where $\psi\in C^2(\Sph^{d-1},\R^*_{+})$. $\frac{1}{N}\sum_{i=1}^N\calE_{t\ell^{-2}}^i\delta_{i/N}$ converges to the solution of
\begin{equation*}
\left\{
\begin{array}{l}
  \partial_t e_t = \bar\gamma \frac{c_\gamma}{2} \partial^2_{xx} e_t \\
  e_0(r) = \int_{E} \left(\frac{1}{2}|v|^2 + U(x) + \frac{1}{2}|x|^2\right)d\mu_0(r,x,v),
\end{array}
\right.
\end{equation*}
where $c_\gamma = \frac{1}{2}\int_{-1/2}^{1/2}u^2\gamma(u) du$, in the sense that for any $g\in \calC^4(\T)$ and any $t>0$,
$$\lim_{\ell\to 0} \lim_{N\to \infty} \sup_{t\leq T}  \E\left[\left | \frac{1}{N}\sum_{i=1}^N \calE^i_{t\ell^{-2}} g\left(\frac{i}{N}\right) - \int_{\T} e_t(r) g(r) dr \right | \right] = 0.$$
\end{Prop}

The potentials $U$ in Proposition \ref{Energy evolution} are exactly $C^2$ homogeneous functions of degree 2, which satisfy \ref{H1} and for which one can prove an equipartition theorem to close the diffusion equation for the energy. In particular, we will prove that one can construct a function $\ell = \ell(N) = c(\log N)^{-1/2}$ for some constant $c>0$, such that
\begin{equation}\label{limit with ell(N)}\lim_{N\to \infty} \sup_{t\leq T}  \E\left[\left | \frac{1}{N}\sum_{i=1}^N \calE^i_{t\ell(N)^{-2}} g\left(\frac{i}{N}\right) - \int_{\T} e_t(r) g(r) dr \right | \right] = 0.\end{equation}



\section{Mean field limit}\label{Section 3}


To define fully define the system \eqref{original system solid new formalism} with the correlation structure of the different Poisson random measures, we use a construction similar to \cite{Cortez} for the Kac model. Instead of considering a collection of Poisson random measure on $\R_+\times \R^d$ that select velocities as in \eqref{original system solid new formalism}, we first define a global Poisson random measure $\calN$ on $\R_+\times\T^2$ that selects the pairs of particles which exchange velocities. $\calN$ selects points in the periodic square $\T^2$ with intensity
\begin{equation}\label{driving Poisson point process}
\bar\gamma N^2 \sum_{i=1}^N\left(\gamma_0 \mathbbm{1}_{(r,r')\in\Lambda^i\times \Lambda^i} + \frac{1}{2}\sum_{\substack{k=-\ell N \\ k\neq 0}}^{\ell N}\gamma_{k} \mathbbm{1}_{(r,r')\in\Lambda^i\times \Lambda^{i+k}} \right)dtdrdr',
\end{equation}
where $\Lambda^i = [(i-1/2)/N, (i+1/2)/N]$ for $-N < i \leq N$. Note that we identify $\Lambda^{i}$ with $\Lambda^{i-N}$ for any $1\leq i \leq N$ as we work in the torus. The velocity exchanges between two different particles with indices $i$ and $i+k$ only happen when the Poisson random measure hits either $\Lambda^i\times\Lambda^{i+k}$ or $\Lambda^{i+k}\times\Lambda^i$, which happens at rate $\bar\gamma \gamma_k$.

We associate to the particle indexed by $i$ a Poisson random measure $\calN^i$ on $\R_+\times \T$ that only selects the velocity exchanges between this particle and one of its neighbours:
\begin{equation}\label{definition Ni} \calN^i(dt,dr) = \calN(dt,\Lambda^i,dr) + \calN(dt,dr,\Lambda^i) - \calN(dt,\Lambda^i,\Lambda^i). \end{equation}
$\calN^i$ has thus intensity
\begin{equation}\label{intensity practical Poisson}
\bar \gamma N \sum_{k=-\ell N}^{\ell N}\gamma_{k} \mathbbm{1}_{r\in \Lambda^{i+k}} dtdr.
\end{equation}

\begin{center}
\begin{figure}[h]
\center
\def\svgwidth{0.35\textwidth}
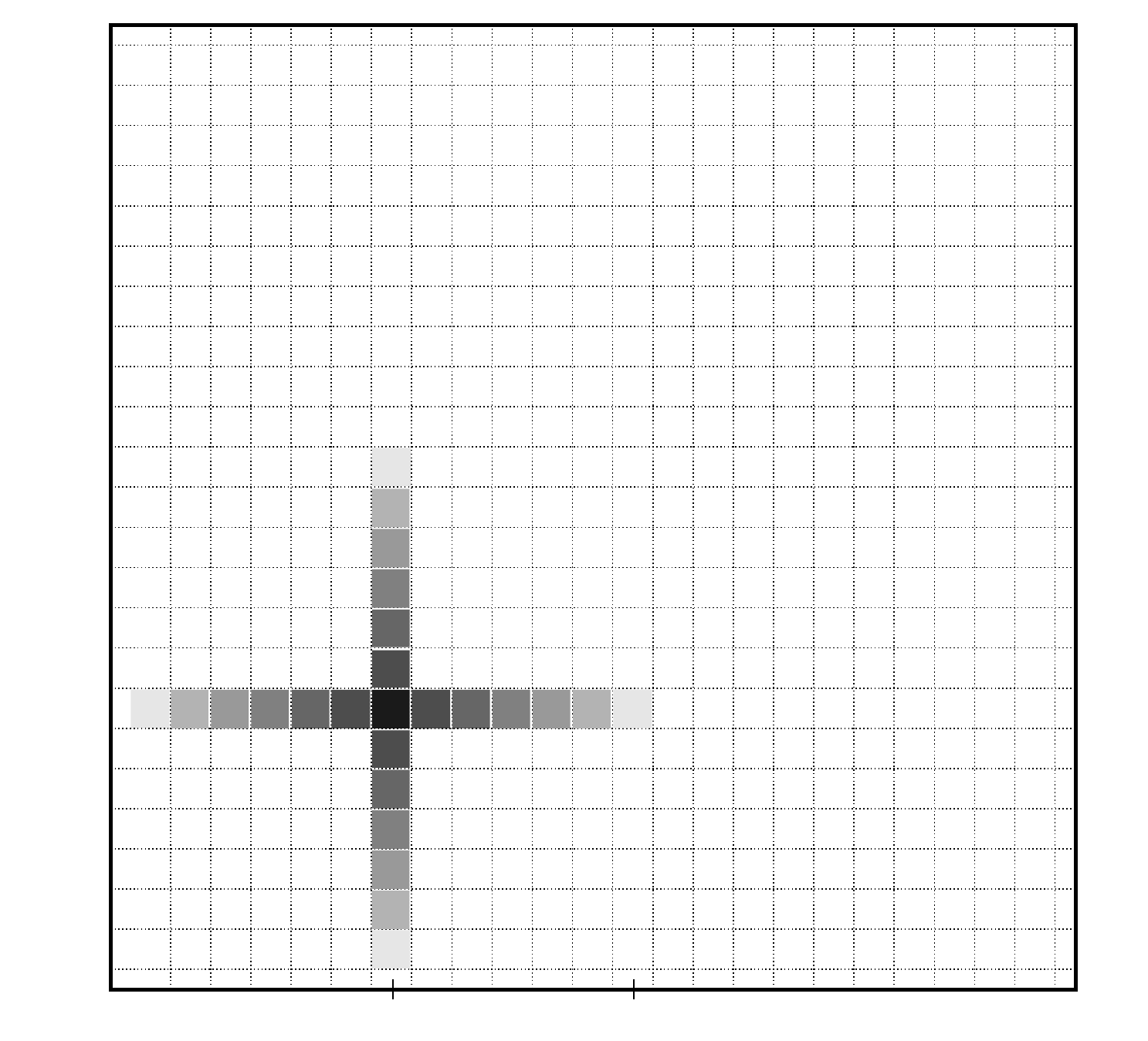
\caption{The colors indicate the intensity of the Poisson random measure $\calN$ over the line  $\T\times\Lambda^i$ and the column $\Lambda^i\times\T$, from which we construct the Poisson random measure $\calN^i$. In this figure, the random measure hits the square $\Lambda^i\times\Lambda^{i+2}$, which thus produces an exchange of velocity between particles $i$ and $i+2$.}
\end{figure}
\end{center}

With these notations, the evolution of the coordinates $(X^i_t,V^i_t)$ for any $1\leq i \leq N$ is given by:
\begin{equation}\label{original system solid}
\left\{
\begin{array}{l}
  dX^i_t = V^i_t dt \\
  dV^i_t = -\left(\sum_{k=-\ell N}^{\ell N} \phi_k \nabla W(X^i_t-X^{i+k}_t) + \nabla U(X^i_t) \right)dt + \int_{\T}(V^{[ N r' ]}_t - V^i_{t^-})d\mathcal{N}^i(t,r').
\end{array}
\right.
\end{equation}
We denote by the brackets in $[ N r' ]$ the rounded value to the nearest integer of $N r'$, which might be bigger or smaller than $N r'$, and with any arbitrary convention for half-integers. The time intensity of $\calN^i$ being bounded by $\bar \gamma$, this evolution is well defined for any $N$ under hypothesis \ref{H1} and \ref{H2} on the potentials. Adapting the construction from \cite{Cortez}, we now define a family of processes $(\widetilde Y^i)_{i\leq N}$ coupled to the original system $(Y^i)_{i\leq N}$, by
\begin{equation}\label{nonlinear process definition}
\left\{
\begin{array}{l}
  d\widetilde X^i_t = \widetilde V^i_t dt \\
  d\widetilde V^i_t = -\left(\int_{\T\times E} \Phi_\ell(\widetilde r^i_0-r')\nabla W(\widetilde X^i_t-x')d\mu_t(y')+ \nabla U(\widetilde X^i_t)\right)dt \\
  \hspace{9cm} + \int_{\T} (\Pi^i_t(r') - \widetilde V^i_{t^-})d\calN^{i}(t,r'),
\end{array}
\right.
\end{equation}
where $\mu$ is the solution of the nonlinear martingale problem \eqref{nonlinear mean field martingale problem}. $\widetilde Y^i$ and $Y^i$ are thus driven by the same Poisson random measure $\calN^i$ for any $i\leq N$. $\Pi^i$ is a measurable mapping that will be precisely defined in Lemma \ref{Optimal Coupling solid}, in such a way that when $\calN^i$ selects a neighbor, the update of velocity for $Y^i$ and $\widetilde Y^i$ is close enough. $\Pi^i_t(r')$ should be actually also written as a function $\Pi^i_t(r',\widetilde r^i_0, \bold V^i_{t_-})$ of $\widetilde r^i_0$ and of the vector $\bold V^i_{t_-} := (V^{i+k}_{t_-})_{-\ell N \leq k \leq \ell N}$, but we omit this dependence for notational convenience. Also, notice that the spatial parameters $r^i$ and $\tilde r^i$ are constant and equal to their initial values $r^i_0$ and $\tilde r^i_0$. We call $(\widetilde Y^i)_{i\leq N}$ the \textit{nonlinear processes}.

It just remains to choose the initial distribution of $(\widetilde Y^i)_{i\leq N}$. Let us first comment on the proof strategy before defining $(\widetilde Y^i_0)_{i\leq N}$. By triangular inequality,
\begin{equation}\label{goal bound}\E\left[\calW_1(\mu^N_t,\mu_t)\right] \leq \E\left[\calW_1(\mu^N_t,\widetilde\mu^N_t)\right] + \E\left[\calW_1(\widetilde\mu^N_t, \mu_t)\right].\end{equation}
where $\widetilde \mu^N_t = 1/N\sum\delta_{Y^i_t}$ is the empirical measure associated with the nonlinear processes. The first term on the right-hand side can be controlled by the coupling, while we need to prove an instance of law of large numbers for empirical measures to control the second term. To prove a law of large numbers result, we face two major issues. First, the system $(\widetilde Y^i)_{i\leq N}$ is strongly correlated since the Poisson random measures $(\calN^i)_{i\leq N}$ share atoms. Following \cite{Cortez}, we will define a new system of nonlinear \textit{independent} processes $(\bar Y^i)_{i\leq N}$ from $(\widetilde Y^i)_{i\leq N}$ and compare both systems. In particular, we will take the initial conditions $\bar Y^i_0$ and $\widetilde Y^i_0$ to be equal almost surely for any $i\leq N$.

Secondly, to prove a law of large numbers type result for $\bar \mu^N_t$, we need the variables $\bar Y^i_t$ to have a similar enough distribution. But, even if we have some regularity on the $r$-parameter for the initial distribution $\mu_0$ by \ref{H7}, we do not know any regularity property at time $t$. In particular, we dot not know to what extent two nonlinear processes $\bar Y^i$ and $\bar Y^j$ with close spatial parameters $\bar r^i$ and $\bar r^j$ have a similar distribution. Therefore, we will deal with truly identically distributed random variables, by diviging $\T$ into mesoscopic boxes and requiring that nonlinear processes in the same mesoscopic box have the same law.

Let $\epsilon_N$ be a parameter representing the macroscopic size of the aforementioned boxes such that $1/N<\epsilon_N<1$, and assume $\epsilon_N$ goes to $0$ with $N$ at a rate to be defined later on. Assume for simplicity that the parameter $\epsilon_N$ is such that $\epsilon_N^{-1}$ is an integer. We subdivide $\T$ in exactly $\epsilon_N^{-1}$ boxes $(B_j)_{1\leq j\leq \epsilon_N^{-1}}$, where $B_j=]j\epsilon_N,(j+1)\epsilon_N]$, for $1\leq j \leq \epsilon_N^{-1}$. We add the extra technical assumption that the integer $\epsilon_N^{-1}$ divides $N$, so that each box contains exactly $N\epsilon_N$ particles of the original system. Denote by $N B_j $ the set of indices of the particles in the box $B_j$:
$$N B_j = \{j N\epsilon_N+1,j N\epsilon_N+2,...,(j+1) N\epsilon_N  \}.$$

Denote by $\mu^j_t$ the distribution $\mu_t$ conditioned on the fact that the coordinate $r$ belongs to $B_j$:
\begin{equation}\label{conditional measure} \mu^j_t(r,z) = \epsilon_N^{-1} \1_{r\in B_j} \mu_t(r,z).\end{equation}
For any $1\leq j \leq \epsilon_N^{-1}$, the variables $(\widetilde Y_0^i)_{i\in NB_j}$ (and therefore $(\bar Y_0^i)_{i\in NB_j}$ as well) are then taken independent and with distribution $\mu^j_0$. In particular, $\tilde r_0^i$ is a uniformly distributed over $B_j$ for any $i\in NB_j$.

Notice that, in a given box $B_j$, the spatial parameters $(\widetilde r^i)_{i\in NB_j}$ are not ordered as their respective indices $i\in NB_j$ in general. Yet, with this construction, we can always localize $\widetilde r^i$ and get the uniform estimate
\begin{equation}\label{uniform bound axial parameters} |i/N -\widetilde r^i| \leq \epsilon_N. \end{equation}
This estimate, together with energy bounds involving $\nabla W$, will help compensating the fact we do not have the uniform Lipschitz property for $\Phi_\ell(r)\nabla W(x)$. The idea of doing couplings over mesoscopic boxes $(B_j)_{1\leq j \leq \epsilon_N^{-1}}$ will appear several times during the proof due to the local mean field structure of our problem.

\begin{center}
\begin{figure}[h]
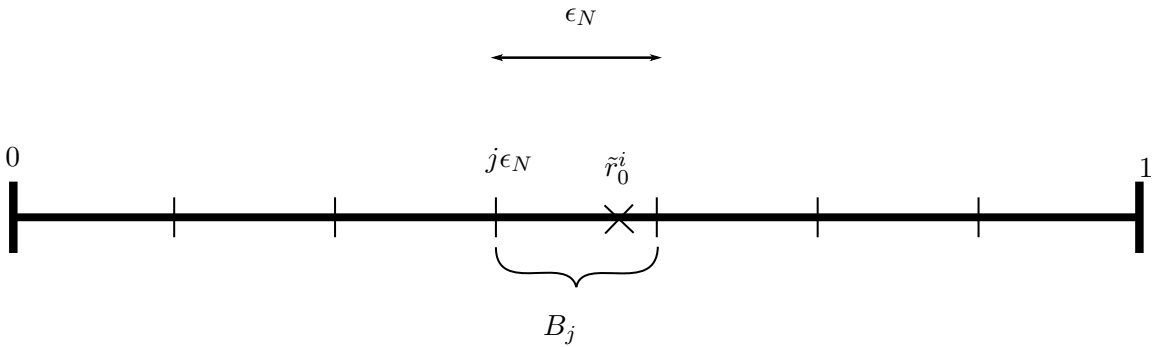
\caption{Boxes $B_j$ of size $\epsilon_N$. For $i\in NB_j$, $\widetilde r^i_0$ is a uniform random variable in $B_j$}
\end{figure}
\end{center}

Before turning to the proof of Theorem \ref{Convergence result}, let us state a series of lemmas that we will use and whose proofs are postponed to the end of the section. In the whole proof and in the lemmas, $K$, $K'$ are constants that change from one line to another that depend on $U$, $W$, $\phi$, $\gamma$, but not on $N$, $\ell$, $\bar \gamma$ since we want to track the dependence of our convergence rates on these terms. Define for any $1\leq i \leq N$ the probability density on $\T$
\begin{equation}\label{measure sigma} \sigma^i(r) = N \sum_{k=-\ell N}^{\ell N}\gamma_{k} \mathbbm{1}_{r\in \Lambda^{i+k}},\end{equation}
from which $\calN^i$ selects a neighbor of particle $i$ to exchange velocities with, see \eqref{intensity practical Poisson}. Define also the weighted distribution 
\begin{equation} \label{weighted distribution} w^{\mu_t,\widetilde r^i_0}(v') = \int_{\T\times\R^d}\Gamma_\ell(\widetilde r^i_0-r')d\mu_t(r',x',v'),\end{equation}
which is the velocity marginal of $\mu_t$, modulated by the function $\Gamma_\ell$ around the axial parameter $\widetilde r^i_0$. This is the distribution we expect for the updates of velocity of the nonlinear processes with axial parameter $\widetilde r^i_0$. In the following lemma, adapted from \cite{Cortez}, we construct the map $\Pi^i$ introduced in \eqref{nonlinear process definition}.

\begin{Lem}\label{Optimal Coupling solid}
Let $1\leq j \leq \epsilon_N^{-1}$ and $i \in NB_j$. Then for any $N$, there exists a measurable mapping
$$\Pi^i: \R_+\times  \T \times B_j \times (\R^d)^{2\ell N+1} \to \R^d,$$
such that, for any $\widetilde r \in B_j$ and any $\bold v \in (\R^d)^{2\ell N +1}$, if $r$ is drawn from \eqref{measure sigma}, then $\Pi^i_t(r,\widetilde r, \bold v)$ has distribution $w^{\mu_t,\widetilde r}$ (see \eqref{weighted distribution}). Moreover, for any coupling $\pi_t$ of $\mu^N_t$ and $\mu_t$, the following bounds hold
\begin{align} \label{bound optimal coupling}
&\frac{1}{N}\sum_{i=1}^N \E\left[ \int_\T | V_t^{[ Nr ]} - \Pi^i_t(r,\widetilde r^i_0,\bold V^i_{t}) | \sigma^i(r) dr \right] \leq K \E\left[ \int_{(\T\times E)^2} | v' - v'' | d \pi_t(y',y'') \right] \nonumber \\
&\hspace{4cm}+ \frac{K}{\ell^2} \E\left[\int_{(\T\times E)^2}|v''| \1_{\{ |\widetilde r_0^i -r''  |<\ell/2\}} \left | r' - r''\right |d\pi_t(y',y'')\right] + \frac{K\epsilon_N}{\ell},
\end{align}\\
and also
\begin{equation} \label{bound optimal coupling 2} \frac{1}{N} \sum_{j=1}^{\epsilon_N^{-1}} \sum_{i\in NB_j} \E\left[\int_{B_j}\left | \Pi^i_t(r,\widetilde r^i_0,\bold V^i_{t}) \right | \sigma^i(r) dr\right] \leq K\frac{\epsilon_N^{1/2}}{\ell^{1/2}}. \end{equation}
\end{Lem}

From \eqref{bound optimal coupling}, we can control the difference of the velocities' updates between the original system and the nonlinear processes. \eqref{bound optimal coupling 2} is a useful bound that we will need. The next lemma shows the marginal distribution at time $t$ of the nonlinear process $\widetilde Y^i$ for $i\in NB_j$ defined in \eqref{nonlinear process definition} is indeed $\mu^j_t$, where $\mu^j_t$ is defined in \eqref{conditional measure}.

\begin{Lem} \label{Nonlinear process construction}
Let $1\leq j \leq \epsilon_N^{-1}$ and $i\in NB_j$. The process $(\widetilde Y^i_t)_{t\geq 0}$ satisfying \eqref{nonlinear process definition} and such that $\widetilde Y^i_0$ has distribution $\mu^j_0$ is well defined and has distribution $\mu^j_t$.
\end{Lem}

The next lemma shows that we can choose a good coupling of the initial conditions of the original system $Y^i_0$ and of the nonlinear processes $\widetilde Y^i_0$.

\begin{Lem} \label{Initial positions}
There exists a coupling of $(Y^i_0)_{i\leq N}$ and $(\widetilde Y^i_0)_{i\leq N}$ such that, for any $1\leq i \leq N$,
$$\E\left[|Y^i_0 - \widetilde Y^i_0|\right] \leq K \epsilon_N,$$
for some constant $K$.
\end{Lem}

This last lemma gives uniform in time moment bounds by conservation type arguments.

\begin{Lem}\label{bound speed}
The solution $\mu$ of the nonlinear martingale problem \eqref{nonlinear mean field martingale problem} satisfies
$$\sup_{t\geq 0} \int_{\T\times E} |z|^2 d\mu_t(r,z) < \infty.$$
For the particle system, the following bound holds
$$\sup_{N\geq 1} \sup_{t\geq 0} \frac{1}{N}\sum_{i=1}^N \E\left[|V^i_t|^2\right] < \infty.$$
\end{Lem}

\subsection{Proof of Theorem \ref{Convergence result}}

\begin{proof}
Recall that in \eqref{goal bound}, we got
$$\E\left[\calW_1(\mu^N_t,\mu_t)\right] \leq \E\left[\calW_1(\mu^N_t,\widetilde\mu^N_t)\right] + \E\left[\calW_1(\widetilde\mu^N_t, \mu_t)\right]$$
by triangular inequality. It suffices to prove the convergence of the two terms on the right-hand side to conclude. Following the terminology in \cite{Cortez}, we will call the first step that consists in bounding $\E[\calW_1(\mu^N_t,\widetilde\mu^N_t)]$ the \textit{coupling} and the second step that consists in bounding $\E[\calW_1(\widetilde\mu^N_t,\mu_t)]$ the \textit{decoupling}.\\

\textit{Step 1: Coupling} \\ 

First, as the two measures $\mu^N_t$ and $\widetilde \mu^N_t$ are atomic, it is easy to bound
$$\calW_1(\mu^N_t,\widetilde\mu^N_t) \leq \frac{1}{N}\sum_{i=1}^N |Y^i_t - \widetilde Y^i_t |,$$
and we are led to bound the expectation of the term on the right-hand side. By the evolution equations \eqref{original system solid} for $(Y^i)_{1\leq i \leq N}$ and \eqref{nonlinear process definition} for $(\widetilde Y^i)_{1\leq i \leq N}$ and the Lipschitz property for $\nabla U$, we can bound:
\begin{align} \label{evolution bound}
\frac{1}{N}\sum_{i=1}^N \E\left[ |Y^i_t - \widetilde Y^i_t |\right ] &\leq \frac{1}{N}\sum_{i=1}^N \E\left[ | Y^i_0 - \widetilde Y^i_0 | \right]+ \int_0^t  \frac{1}{N}\sum_{i=1}^N \E\left[|Y^i_s - \widetilde Y^i_s| \right] ds \nonumber \\
     &\quad\quad + \int_0^t ds \frac{1}{N}\sum_{i=1}^N \E \Bigg[\Bigg | \int_{\T\times E} \Phi_\ell\left(\frac{i}{N} - r'\right) \nabla W(X^i_s - x') d\mu^N_s(r',z') \nonumber \\
             &\quad\quad\quad\quad\quad\quad\quad\quad\quad\quad\quad\quad- \int_{\T\times E} \Phi_\ell(\widetilde r^i_0 - r'') \nabla W(\widetilde X^i_s - x'') d\mu_s(r'',z'') \Bigg | \Bigg] \nonumber \\
     &\quad\quad + \frac{1}{N}\sum_{i=1}^N \E \left[  \int_0^t\int_{\T} |V^{[ Nr ]} - \Pi^i_s(r) - \left(V^i_{s_-} - \widetilde V^i_{s_-} \right)| d\calN^i(s,r) \right].
\end{align}
Denoting by $\pi_s$ a coupling between $\mu^N_s$ and $\mu_s$, the force term involving $\nabla W$ in \eqref{evolution bound}, can be bounded by:
\begin{align*}
&\frac{1}{N}\sum_{i=1}^N\E\left[\int_{(\T\times E)^2} \left | \Phi_\ell\left(\frac{i}{N} - r'\right) \nabla W(X^i_s - x') - \Phi_\ell(\widetilde r^i_0 - r'') \nabla W(\widetilde X^i_s - x'') \right | d\pi_s(y',y'')  \right] \\
   &\quad\quad \leq \frac{1}{N}\sum_{i=1}^N \E\left[ \int_{\T\times E} \Phi_\ell\left(\frac{i}{N} - r'\right) \left | \nabla W(X^i_s - x') - \nabla W(\widetilde X^i_s - x') \right |  d\mu^N_s(y') \right] \\
      &\quad\quad\quad + \frac{1}{N}\sum_{i=1}^N \E\left[\int_{(\T\times E)^2} \Phi_\ell\left(\frac{i}{N} - r'\right)\left |  \nabla W(\widetilde X^i_s - x')  -  \nabla W(\widetilde X^i_s - x'') \right |d\pi_s(y',y'') \right] \\
      &\quad\quad\quad +\frac{1}{N}\sum_{i=1}^N \E\left[\int_{(\T\times E)^2} \left | \Phi_\ell\left(\frac{i}{N} - r'\right) - \Phi_\ell\left(\frac{i}{N} - r''\right) \right |\left | \nabla W(\widetilde X^i_s - x'') \right |d\pi_s(y',y'') \right] \\
       &\quad\quad\quad+\frac{1}{N}\sum_{i=1}^N \E\left[\int_{(\T\times E)^2} \left | \Phi_\ell\left(\frac{i}{N} - r''\right) - \Phi_\ell(\widetilde r^i_0 - r'') \right |\left | \nabla W(\widetilde X^i_s - x'') \right |d\mu_s(y'') \right]\\
       &\quad\quad \leq A_1 + A_2 + A_3 + A_4,
\end{align*}
by introducing the four cross terms in the second line and denoting by $A_k$ the term numbered $1\leq k\leq 4$ at the last line. By the Lipschitz property for $\nabla W$, the first term is bounded by
\begin{equation}\label{term 1} A_1 \leq K \frac{1}{N}\sum_{i=1}^N \sum_{k=-\ell N}^{\ell N} \Phi_\ell\left(\frac{i-k}{N}\right) \E\left[ \left | X^i_s- \widetilde X^i_s \right | \right] \leq K'  \frac{1}{N}\sum_{i=1}^N \E\left[ \left | X^i_s- \widetilde X^i_s \right | \right], \end{equation}
since $\sum_k\Phi_\ell(k/N)$ is bounded. Similarly, the second term is bounded by 
\begin{equation}\label{term 2} A_2 \leq K'  \E\left[\int_{(\T\times E)^2} \left | x'   - x'' \right |d\pi_s(y',y'') \right]. \end{equation}
It remains to choose a good coupling to bound this term. The fourth term is also relatively easy to bound by the Lipschitz property for $\Phi_\ell$, which has Lipschitz constant equal to $Lip(\phi)\ell^{-2}$:
\begin{align}\label{term 4}
A_4&\leq \frac{K}{\ell^2}\frac{1}{N}\sum_{i=1}^N \E\left[\int_{(\T\times E)^2} \left | \frac{i}{N}  - \widetilde r^i_0  \right | \1_{| \widetilde r^i_0-r''|<\ell/2} \left | \nabla W(\widetilde X^i_s - x'') \right |d\mu_s(y'') \right] \nonumber \\
      &\leq \frac{K\epsilon_N}{\ell^2}\frac{1}{N}\sum_{i=1}^N \E\left[\int_{(\T\times E)^2}  \1_{| \widetilde r^i_0-r''|<\ell/2} \left | \nabla W(\widetilde X^i_s - x'') \right |d\mu_s(y'') \right] \nonumber \\
      &\leq \frac{K'\epsilon_N}{\ell}.
\end{align}
We used at the first line that $\Phi_\ell$ has support included in $[-\ell/2, \ell/2]$ to introduce the indicator function, and we used at second line the uniform bound \eqref{uniform bound axial parameters}. For the last line, we used the following lemma. 

\begin{Lem}\label{Lemma bound potential energy}
Under hypothesis \ref{H1} on $W$, the following uniform in time bound holds
$$\sup_{0<\ell <1} \sup_{N\geq 1}\sup_{t\geq 0} \frac{1}{N}\sum_{i=1}^N \E\left[\int_{\T\times E} \frac{1}{\ell} \1_{| \widetilde r^i_0-r''|<\ell/2} \left | \nabla W(\widetilde X^i_t - x'') \right |d\mu_t(y'') \right] <\infty.$$
\end{Lem}

The proof of Lemma \ref{Lemma bound potential energy} is postponed to the end of the section. Using the same arguments, the third term is bounded by
\begin{equation} \label{term 3} A_3\leq \frac{K}{\ell^2}\frac{1}{N}\sum_{i=1}^N \E\left[\int_{(\T\times E)^2} \left | r' - r'' \right | \1_{| \widetilde r^i_0-r''|<\ell/2}\left | \nabla W(\widetilde X^i_s - x'') \right |d\pi_s(y',y'') \right].
\end{equation}

It just remains to choose an appropriate coupling $\pi_s$ of $\mu^N_s$ and $\mu_s$ to bound \eqref{term 2} and \eqref{term 3}, keeping in mind that we will need to apply Gronwall's lemma at some point. A first option would be to take $\pi_s$ to be the optimal coupling, which would give the best bound for \eqref{term 2}, but this gives no information on the term \eqref{term 3} since $\nabla W$ is unbounded. The local mean field structure of the system we consider actually leads us to introduce again couplings over mesoscopic boxes. More precisely, recall from \eqref{conditional measure} the notation $\mu^j_s$ for the measure $\mu_s$ conditioned on the fact that $r\in B_j$. We similarly define the empirical probability measure on the box $B_j$:
$$\mu^{N,j}_s := \frac{1}{N\epsilon_N} \sum_{i\in NB_j} \delta_{Y^i_s}.$$
Then, denoting by $\pi^j_s$ the optimal coupling between $\mu^{N,j}_s$ and $\mu^j_s$ for $\calW_1$, it is easy to see that
$$\pi_s := \epsilon_N\sum_{j=1}^{\epsilon_N^{-1}} \pi^j_s$$
is a coupling of $\mu^N_s$ and $\mu_s$. In words, this coupling consists in choosing a box $B_j$ uniformly at random over $(B_j)_{1\leq j \leq \epsilon_N^{-1}}$ and then, given that the coordinates $r',r'' \in B_j$, select the optimal coupling $\pi^j$. With this approach, the term \eqref{term 2} can then be bounded by
$$A_2 \leq K'  \E\left[\int_{(\T\times E)^2} \left | x'   - x'' \right |d\pi_s(y',y'') \right] \leq K' \epsilon_N \sum_{j=1}^{\epsilon_N^{-1}} \E\left[ \calW_1(\mu^{N,j}_s,\mu^j_s)\right],$$
and the term \eqref{term 3} is bounded by
\begin{align*}
A_3& \leq K \frac{\epsilon_N}{\ell^2} \frac{1}{N}\sum_{i=1}^N \epsilon_N \sum_{j=1}^{\epsilon_N^{-1}} \E\left[\int_{(\T\times E)^2} \1_{| \widetilde r^i_0-r''|<\ell/2}\left | \nabla W(\widetilde X^i_s - x'') \right |d\mu^j_s(y'') \right] \\
  &\leq K \frac{\epsilon_N}{\ell^2} \frac{1}{N}\sum_{i=1}^N \E\left[\int_{(\T\times E)^2} \1_{| \widetilde r^i_0-r''|<\ell/2}\left | \nabla W(\widetilde X^i_s - x'') \right |d\mu_s(y'') \right] \\
  & \leq K' \frac{\epsilon_N}{\ell},
\end{align*}
by Lemma \ref{Lemma bound potential energy}. All in all, assembling the last two inequalities and the bounds \eqref{term 1} and \eqref{term 4}, the force terms involving $\nabla W$ and $\nabla U$ in \eqref{evolution bound} are bounded by
\begin{equation} \label{bound force terms} K\int_0^t\left( \frac{1}{N}\sum_{i=1}^N \E\left[ \left | X^i_s- \widetilde X^i_s \right | \right]  + \epsilon_N \sum_{j=1}^{\epsilon_N^{-1}} \E\left[ \calW_1(\mu^{N,j}_s,\mu^j_s)\right]  + \frac{\epsilon_N}{\ell} \right)ds.
\end{equation}
Finally, it remains to control the terms coming from the exchanges of velocities. Recall that $\calN^i$ has intensity $\bar\gamma \sigma^i(r)drdt$, with $\sigma^i$ defined in \eqref{measure sigma}. Since $(t,r)\mapsto \Pi^i_t(r)$ is measurable, one can bound
\begin{align*}
B &:= \frac{1}{N}\sum_{i=1}^N \E\left[ \int_0^t |V^{[ Nr ]} - \Pi^i_s(r) - \left(V^i_{s_-} - \widetilde V^i_{s_-} \right)| d\calN^i(s,r)\right] \\
& \hspace{0.8cm} \leq \bar\gamma \frac{1}{N}\sum_{i=1}^N \E\left[ \int_0^t ds\int_{\T} \left|V^{[ Nr ]} - \Pi^i_s(r)\right | \sigma^i(r)dr\right] + \bar\gamma \frac{1}{N}\sum_{i=1}^N \E\left[\int_0^tds\int_{\T} \left |V^i_s - \widetilde V^i_s \right | d\sigma^i(r)dr\right] \\
& \hspace{0.8cm} \leq \bar\gamma \int_0^t  \frac{1}{N}\sum_{i=1}^N \E\left[ \int_{\T} \left|V^{[ Nr ]} - \Pi^i_s(r)\right | \sigma^i(r)dr\right] ds + \bar\gamma \int_0^t \frac{1}{N}\sum_{i=1}^N \E\left[ \left |V^i_s - \widetilde V^i_s \right | \right] ds.
\end{align*}
By \eqref{bound optimal coupling} in Lemma \ref{Optimal Coupling solid}, the first term is bounded by:
\begin{align*}
&K\bar\gamma \int_0^t ds \E\left[ \int_{(\T\times E)^2} | v' - v'' | d \pi_s(y',y'') \right] \\
      &\hspace{2cm}+ K \int_0^t \left(\E\left[\int_{(\T\times E)^2}|v''|\frac{1}{\ell^2} \1_{\{ |\widetilde r_0^i -r''  |<\ell/2\}} \left | r' - r''\right |d\pi_s(y',y'')\right] + \frac{\epsilon_N}{\ell} \right)ds,
\end{align*}
for any coupling $\pi_s$ of $\mu^N_s$ and $\mu_s$. Once again we choose the coupling $\pi_s = \epsilon_N\sum_{j=1}^{\epsilon_N^{-1}} \pi^j_s$ over the mesoscopic boxes $(B_j)_{1\leq j \leq \epsilon_N^{-1}}$ and we get the bound:
$$K\bar\gamma \epsilon_N \int_0^t\left( \sum_{j=1}^{\epsilon_N^{-1}} \E\left[ \calW_1(\mu^{N,j}_s,\mu^j_s)\right] + \frac{1}{\ell} \int_{\T\times E} |v''| d\mu_s+ \frac{1}{\ell}\right)ds.$$
Using Lemma \ref{bound speed} to bound $\int_{\T\times E} |v''| d\mu_s$ by a constant, we finally can bound the terms coming from the exchanges of velocity in \eqref{evolution bound} by 
\begin{equation}\label{bound exchanges of velocities}B \leq K\bar\gamma \int_0^t \left( \frac{1}{N}\sum_{i=1}^N \E\left[ \left |V^i_s - \widetilde V^i_s \right | \right]  + \epsilon_N \sum_{j=1}^{\epsilon_N^{-1}} \E\left[ \calW_1(\mu^{N,j}_s,\mu^j_s)\right] + \frac{\epsilon_N}{\ell} \right) ds. \end{equation}

Combining the bound on the force terms \eqref{bound force terms} and the bound on the exchanges of velocities \eqref{bound exchanges of velocities}, we obtain
\begin{align}\label{intermediate bound}
\frac{1}{N}\sum_{i=1}^N \E\left[ |Y^i_t - \widetilde Y^i_t |\right ]  &\leq \frac{1}{N}\sum_{i=1}^N \E\left[ | Y^i_0 - \widetilde Y^i_0 | \right] \\
    &\hspace{0.1cm}+ K (1+\bar\gamma) \int_0^t \left( \frac{1}{N}\sum_{i=1}^N \E\left[ \left |Y^i_s - \widetilde Y^i_s \right | \right]  +  \epsilon_N \sum_{j=1}^{\epsilon_N^{-1}} \E\left[ \calW_1(\mu^{N,j}_s,\mu^j_s)\right] + \frac{\epsilon_N}{\ell}  \right) ds. \nonumber
\end{align}

As our goal is to ultimately combine this inequality with \eqref{goal bound}, we can introduce the empirical measure associated with the nonlinear processes $(\widetilde Y^i_s)_{i\in N B_j}$ to replace the term $\calW_1(\mu^{N,j}_s,\mu^j_s)$. Define
$$\widetilde \mu^{N,j}_s = \frac{1}{N\epsilon_N}\sum_{i\in NB_j} \delta_{\widetilde Y^i_s},$$
the obvious analogue of $\mu^{N,j}_s$ for the system of nonlinear processes $(\widetilde Y^i_s)_{i\in N B_j}$. We then apply the triangular inequality and a straightforward choice of coupling between $\mu^{N,j}_s$ and $\widetilde \mu^{N,j}_s$ to bound
$$\calW_1(\mu^{N,j}_s,\mu^j_s) \leq \calW_1(\mu^{N,j}_s,\widetilde \mu^{N,j}_s) + \calW_1(\widetilde \mu^{N,j}_s,\mu^j_s) \leq \frac{1}{N\epsilon_N} \sum_{i\in NB_j} |Y^i_s - \widetilde Y^i_s | +  \calW_1(\widetilde \mu^{N,j}_s,\mu^j_s).$$
Summing this last inequality over $j$, we obtain the mean value $1/N\sum_{i=1}^N |Y^i_s - \widetilde Y^i_s|$ for the first term on the right-hand side. Inserting it in \eqref{intermediate bound}, we finally get 
\begin{align}\label{final bound coupling}
\frac{1}{N}\sum_{i=1}^N \E\left[ |Y^i_t - \widetilde Y^i_t |\right ]  &\leq \frac{1}{N}\sum_{i=1}^N \E\left[ | Y^i_0 - \widetilde Y^i_0 | \right]  \\
   &\hspace{0.1cm}+K (1+\bar\gamma) \int_0^t \left( \frac{1}{N}\sum_{i=1}^N \E\left[ \left |Y^i_s - \widetilde Y^i_s \right | \right]  +  \epsilon_N \sum_{j=1}^{\epsilon_N^{-1}} \E\left[ \calW_1(\widetilde \mu^{N,j}_s,\mu^j_s)\right] + \frac{\epsilon_N}{\ell}  \right) ds. \nonumber
\end{align}\\

\textit{Step 2: Decoupling} \\ 

Observing \eqref{goal bound} and \eqref{final bound coupling}, and since
$$\E\left[ \calW_1(\widetilde \mu^N_t,\mu_t)\right] \leq \epsilon_N \sum_{j=1}^{\epsilon_N^{-1}} \E\left[ \calW_1(\widetilde \mu^{N,j}_t,\mu^j_t)\right],$$
we can see that it suffices to control the right-hand side term
$$\epsilon_N \sum_{j=1}^{\epsilon_N^{-1}-1} \E\left[ \calW_1(\widetilde \mu^{N,j}_t,\mu^j_t)\right]$$
to conclude the proof. As noted earlier, the sequence of variables $(\widetilde Y^i_t)_{i\leq N}$ is correlated since particles with close indices $i,i'$ are driven by two Poisson $\calN^i$ and $\calN^{i'}$ that share some atoms. Following \cite{Cortez}, we introduce a new system $(\bar Y^i)_{i\leq N}$ of nonlinear processes by replacing those shared atoms. In particular, we take advantage of our box structure $(B_j)_{1\leq j \leq \epsilon_N^{-1}}$ to only require independence of $(\bar Y^i_t)_{i\in NB_j}$ on each box.

The construction of the nonlinear processes $(\bar Y^i_t)_{1\leq i \leq N}$ is the following. Recall that $\calN^i$ was defined from a macroscopic Poisson random measure $\calN$ by \eqref{definition Ni}. We define a Poisson random measure $\calM$ is on $\R_+\times \T^2$, independent of $\calN$, with same intensity given by \eqref{driving Poisson point process}. For $1\leq j \leq \epsilon_N^{-1}$ and $i\in NB_j$, we define $\calM^i$ by
$$\calM^i(dt,dr) = \calN(dt,\Lambda^i,dr) + \calN(dt,dr,\Lambda^i)\1_{r\notin B_j} + \calM(dt,dr,\Lambda^i)\1_{r\in B_j\setminus \Lambda^i}.$$
By their definitions, the Poisson random measures $\calM^i$ and $\calM^{i'}$ do not share any atom if $i,i'$ are in the same set $NB_j$ for some $j$. The processes $(\bar Y^i_t)_{1\leq i \leq N}$ are now defined by the same equations \eqref{nonlinear process definition} as $(\widetilde Y^i_t)_{1\leq i \leq N}$, but replacing $\calN^i$ by $\calM^i$, and with initial values $\bar Y^i_0 = \widetilde Y^i_0$.

\begin{Lem}\label{lemma independent processes}
For any $1\leq j \leq \epsilon_N^{-1}$, the processes $(\bar Y^i)_{i\in NB_j}$, defined by
\begin{equation}\label{definition independent nonlinear processes}
\left\{
\begin{array}{l}
  d\bar X^i_t = \bar V^i_t dt \\
  d\bar V^i_t = -\left(\int_{\T\times E} \Phi_\ell(\bar r^i_0-r')\nabla W(\widetilde X^i_t-x')d\mu_t(y')+ \nabla U(\bar X^i_t)\right)dt + \int_{\T} (\Pi^i_t(r') - \bar V^i_{t^-})d\calM^{i}(t,r') \\
  \bar Y^i_0 = \widetilde Y^i_0,
\end{array}
\right.
\end{equation}
are independent.
\end{Lem}

Denoting by $\bar \mu^{N,j}_t$ the empirical measure associated with $(\bar Y^i_t)_{i\in NB_j}$, the triangular inequality gives
\begin{equation}\label{bound decoupling 5}\E\left[\calW_1(\widetilde\mu^{N,j}_t,\mu^j_t)\right] \leq \E\left[\calW_1(\widetilde\mu^{N,j}_t,\bar\mu^{N,j}_t)\right] + \E\left[\calW_1(\bar\mu^{N,j}_t,\mu^j_t)\right],\end{equation}
and it suffices to control $\E[\calW_1(\widetilde\mu^{N,j}_t,\bar\mu^{N,j}_t)]$ to conclude, the other term being controlled by the law of large numbers result in Appendix \ref{Appendix A}. The bound
\begin{equation} \label{easy bound wasserstein} \calW_1(\widetilde\mu^j_t,\bar\mu^j_t) \leq \frac{1}{N\epsilon_N} \sum_{i\in NB_j} |\widetilde Y^i_t - \bar Y^i_t |\end{equation}
is easily obtained for all $1\leq j\leq \epsilon_N^{-1}$. As we are ultimately interested in founding a bound for the expectation and mean over $j$ of $\calW_1(\widetilde\mu^j_t,\bar\mu^j_t)$, we are finally led to bound
\begin{align} \label{bound decoupling}
\frac{1}{N}&\sum_{i=1}^N \E\Big[|\widetilde Y^i_t - \bar Y^i_t | \Big] \leq K\int_0^t ds \frac{1}{N}\sum_{i=1}^N \E\left[|\widetilde Y^i_s - \bar Y^i_s | \right] \nonumber \\
&\quad+ \int_0^t ds \frac{1}{N}\sum_{i=1}^N \E\left[\int_{\T\times E} \Phi_\ell\left(\widetilde r^i_0 -r'\right) \left|\nabla W(\widetilde X^i_s-x') - \nabla W(\bar X^i_s-x')\right | d\mu_s(r',z')\right] \nonumber \\
&\quad+ \frac{1}{N}\sum_{i=1}^N \E\left[ \left | \int_0^t \int_{\T} \left(\Pi^i_s(r) - \widetilde V^i_{s_-}\right) d\calN^i(s,r) - \int_0^t \int_{\T} \left(\Pi^i_s(r) - \bar V^i_{s_-}\right) d\calM^i(s,r) \right | \right],
\end{align}
by the Lipschitz property for $\nabla U$. Notice that we used that $\widetilde r_0^i = \bar r_0^i$ for the term at the second line. Therefore, by the Lipchitz property for $\nabla W$, we obtain the bound
\begin{equation}\label{bound decoupling 1} K' \int_0^t ds \frac{1}{N}\sum_{i=1}^N \E\left[| \widetilde Y^i_s - \bar Y^i_s | \right] \end{equation}
for the first two terms in \eqref{bound decoupling}. The last term in \eqref{bound decoupling} can be bounded by splitting it in the three following terms:
\begin{align*}
C_1 + C_2 + C_3 &= \frac{1}{N}\sum_{j=1}^{\epsilon_N^{-1}}\sum_{i\in NB_j} \E\bigg[ \int_0^t \int_{\T} \left | \widetilde V^i_{s_-}- \bar V^i_{s_-}\right | \left(d\calN(s,\Lambda^i,r) +d\calN(s,r,\Lambda^i)\1_{r\notin B_j}\right) \bigg] \\
  &\hspace{2cm} + \frac{1}{N}\sum_{j=1}^{\epsilon_N^{-1}}\sum_{i\in NB_j} \E\left[  \int_0^t \int_{\T} \left | \Pi^i_s(r) - \widetilde V^i_{s_-}\right | d\calN(s,r,\Lambda^i)\1_{r\in B_j\setminus \Lambda^i} \right] \\
  &\hspace{2cm} + \frac{1}{N}\sum_{j=1}^{\epsilon_N^{-1}}\sum_{i\in NB_j} \E\left[  \int_0^t \int_{\T} \left | \Pi^i_s(r) - \bar V^i_{s_-}\right | d\calM(s,r,\Lambda^i)\1_{r\in B_j\setminus \Lambda^i} \right],
\end{align*}
where the first term $C_1$ corresponds to the simultaneous jumps of $\widetilde V^i$ and $\bar V^i$, \textit{i.e.} the shared atoms of $\calN^i$ and $\calM^i$ for $1\leq i \leq N$, and the two other terms $C_2$ and $C_3$ respectively correspond to the jumps of $\widetilde Y^i$ alone and to the jumps of $\bar Y^i$ alone. Computing the expectation with respect to the Poisson integral in $C_1$, we get:
\begin{equation}\label{bound decoupling 2} C_1 =  \frac{1}{N} \sum_{j=1}^{\epsilon_N^{-1}}\sum_{i\in NB_j} \frac{\bar\gamma}{2} \left(1 + \sum_{k\,:\,i+k\notin N B_j} \gamma_k \right)  \E\left[ \int_0^t \left | \widetilde V^i_s- \bar V^i_s\right |  ds \right] \leq \bar\gamma \int_0^t ds \frac{1}{N}\sum_{i=1}^N  \E\left[ \left | \widetilde V^i_s- \bar V^i_s\right |\right].
\end{equation}
As for the term $C_2$ coming from the jumps of $\calN^i$ alone, one immediately gets:
\begin{align} \label{bound decoupling 4}
C_2 &\leq \bar\gamma \frac{1}{N}\sum_{j=1}^{\epsilon_N^{-1}}\sum_{i\in NB_j}  \E\left[\int_0^t ds\int_{B_j}\left(\left | \Pi^i_s(r) \right | + \left | \widetilde V^i_s \right |\right) \sigma^i(r) dr\right] \nonumber \\
&\leq \bar\gamma \frac{1}{N}\sum_{j=1}^{\epsilon_N^{-1}}\sum_{i\in NB_j} \E\left[\int_0^t ds\int_{B_j}\left | \Pi^i_s(r) \right | \sigma^i(r) dr\right] + \bar\gamma \sum_{k= -N\epsilon_N}^{N\epsilon_N} \gamma_k \int_0^t \frac{1}{N} \sum_{i=1}^N \E\left[ \left | \widetilde V^i_s \right | \right] ds \nonumber \\
&\leq \bar\gamma \frac{1}{N}\sum_{j=1}^{\epsilon_N^{-1}}\sum_{i\in NB_j} \E\left[\int_0^t ds\int_{B_j}\left | \Pi^i_s(r) \right | \sigma^i(r)dr\right] + K \bar\gamma\frac{\epsilon_N}{\ell} \int_0^t \frac{1}{N} \sum_{i=1}^N \E\left[ \left | \widetilde V^i_s \right | \right] ds,
\end{align}
where we used the symmetry and monotonic properties \ref{H2} of $\gamma$ to bound the integral $\int_{B_j}\sigma^i(r)dr$ by the sum  $\sum_{k=-N\epsilon_N}^{N\epsilon_N} \gamma_k$ at the second line, and then bounded this sum by $K\epsilon_N/\ell$ by the definition \eqref{definition gamma} of $\gamma_k$. The sum in the second term of \eqref{bound decoupling 4} is equal to
$$\frac{1}{N} \sum_{i=1}^N \E\left[ \left | \widetilde V^i_s \right | \right] = \sum_{j=1}^{\epsilon_N^{-1}} \int_{B_j\times E} |v| d\mu_s = \int_{\T\times E} |v|d\mu_s<\infty,$$
which is finite by Lemma \ref{bound speed}. By the bound \eqref{bound optimal coupling 2} of Lemma \ref{Optimal Coupling solid}, we can bound the first term in \eqref{bound decoupling 4} and finally get the following bound for $C_2$:
\begin{equation}\label{bound decoupling 3} C_2 \leq Kt\bar\gamma\left(\frac{\epsilon_N^{1/2}}{\ell^{1/2}} + \frac{\epsilon_N}{\ell}\right) \leq K't\bar\gamma\frac{\epsilon_N^{1/2}}{\ell^{1/2}},\end{equation}
for $N$ large enough. The same bound can be obtained for the term $C_3$ coming from the jumps of $\calM^i$ alone. Combining the bound on the force terms \eqref{bound decoupling 1}, the bound on the simultaneous jumps \eqref{bound decoupling 2} and the bounds on the asynchronous jumps \eqref{bound decoupling 3} in \eqref{bound decoupling}, we get
$$\frac{1}{N}\sum_{i=1}^N \E\Big[|\widetilde Y^i_t - \bar Y^i_t | \Big] \leq K\int_0^t  \left((1+\bar\gamma)\frac{1}{N}\sum_{i=1}^N \E\left[| \widetilde Y^i_s - \bar Y^i_s | \right] +\bar\gamma \frac{\epsilon_N^{1/2}}{\ell^{1/2}} \right)ds.$$
By Gronwall's lemma and \eqref{easy bound wasserstein}, we finally get 
\begin{equation} \label{final bound decoupling}
\epsilon_N \sum_{j=1}^{\epsilon_N^{-1}} \E\left[\calW_1(\widetilde\mu^{N,j}_t,\bar\mu^{N,j}_t)\right] \leq \frac{\bar\gamma}{1+\bar\gamma} \frac{\epsilon_N^{1/2}}{\ell^{1/2}} \left(\e^{K(1+\bar \gamma) t}-1\right).
\end{equation}
Since by Lemma \ref{bound speed}, the measure $\mu_t$ admits a moment of order $2$ for all $t\geq 0$, we can apply the results of Appendix \ref{Appendix A} and get
$$\epsilon_N \sum_{j=1}^{\epsilon_N^{-1}}\E\left[\calW_1(\bar\mu^{N,j}_t,\mu^j_t)\right] \leq K' (N\epsilon_N)^{-\frac{1}{4(d+1)}} + K'\epsilon_N.$$
Combining this result with \eqref{final bound decoupling} in \eqref{bound decoupling 5} finally gives 
\begin{equation} \label{very final bound decoupling}\epsilon_N \sum_{j=1}^{\epsilon_N^{-1}} \E\left[ \calW_1(\widetilde \mu^{N,j}_t,\mu^j_t)\right] \leq \frac{\bar\gamma}{1+\bar\gamma} \frac{\epsilon_N^{1/2}}{\ell^{1/2}} \e^{K(1+\bar \gamma) t} + K' (N\epsilon_N)^{-\frac{1}{4(d+1)}}+K'\epsilon_N.
\end{equation}\\

\textit{Conclusion of the proof of Theorem \ref{Convergence result}}\\

Combining \eqref{final bound coupling} and \eqref{very final bound decoupling}, and applying Gronwall's lemma, we get the bound:
\begin{align*}
\frac{1}{N}\sum_{i=1}^N \E\left[ |Y^i_t - \widetilde Y^i_t |\right ] &\leq \frac{1}{N}\sum_{i=1}^N \E\left[ |Y^i_0 - \widetilde Y^i_0 |\right ] \e^{K(1+\bar\gamma)t} \\
     &\hspace{3cm} + K' \left(\frac{\epsilon_N}{\ell} + (N\epsilon_N)^{-\frac{1}{4(d+1)}} + \epsilon_N +\frac{\bar\gamma}{1+\bar\gamma} \frac{\epsilon_N^{1/2}}{\ell^{1/2}}\right) \e^{K(1+\bar\gamma)t}. 
\end{align*}
By Lemma \ref{Initial positions}, we can bound $\E[ |Y^i_0 - \widetilde Y^i_0 |]$ by $\epsilon_N$ for all $i$. Since $\epsilon_N$ is smaller than $\epsilon_N/\ell$, we finally get
$$\frac{1}{N}\sum_{i=1}^N \E\left[ |Y^i_t - \widetilde Y^i_t |\right ] \leq K'\left(\frac{\epsilon_N}{\ell} + (N\epsilon_N)^{-\frac{1}{4(d+1)}}+ \frac{\bar\gamma}{1+\bar\gamma} \frac{\epsilon_N^{1/2}}{\ell^{1/2}}\right) \e^{K(1+\bar\gamma)t}.$$
Finally, combining this bound together with \eqref{very final bound decoupling} in \eqref{goal bound}, gives
$$\E\left[ \calW_1(\mu^N_t,\mu_t) \right] \leq K' \left(\frac{\epsilon_N}{\ell} + (N\epsilon_N)^{-\frac{1}{4(d+1)}} + \frac{\bar\gamma}{1+\bar\gamma} \frac{\epsilon_N^{1/2}}{\ell^{1/2}}\right)\e^{K(1+\bar\gamma)t}.$$

\end{proof}


\subsection{Proof of the intermediate lemmas}

\begin{proof}[Proof of Lemma \ref{Optimal Coupling solid}]
\hspace{0.5cm}\textit{1. Construction of $\Pi^i$}\\

The construction of $\Pi^i$ is based on Lemma 3 in \cite{Cortez}. Let $1\leq j \leq \epsilon_N^{-1}$ and $i\in NB_j$. Fix $\widetilde r \in B_j$ and $\bold v^i = (v^{i+k})_{-\ell N \leq k \leq \ell N} \in (\R^d)^{2\ell N+1}$. Recall from \eqref{weighted distribution} the notation for the weighted measure $w^{\mu_t,\widetilde r }$ around $\widetilde r$, from which we select a new velocity for particle $\widetilde Y^i$. Denote also by $w^{\bold v^i}$ the weighted measure on $\R^d$ corresponding to the update of velocity among $\bold v^i$:
$$w^{\bold v^i} = \sum_{k=-\ell N}^{\ell N}\gamma_k\delta_{v^{i+k}}.$$
From a random variable $r$ with distribution $\sigma^i$ (see \eqref{measure sigma}), we want to construct a couple of variables $(v^{[ Nr ]},\Pi^i_t(r,\tilde r, \bold v^i))$ with respective distributions $w^{\bold v^i}$ and $w^{\mu_t,\widetilde r}$ for all $(t,\tilde r,\bold v^i)$. In addition, we want the distance between those variables to be small in the sense that
\begin{equation} \label{coupling goal} \int_\T | v^{[ Nr ]} - \Pi^i_t(r,\tilde r, \bold v^i) | \sigma^i(r) dr = \calW_1(w^{\bold v^i},w^{\mu_t,\widetilde r}).\end{equation}
It is straightforward to see that $v^{[ Nr ]}$ has indeed distribution $w^{\bold v^i}$ in that case. The construction of a function $\Pi^i(r)$ with distribution $w^{\mu_t,\widetilde r}$ such that \eqref{coupling goal} holds principally amounts to the construction of a random variable with given law from a uniform random variable.

In fact, let $p^i_{t,\tilde r, \bold v^i}(v,\tilde v)$ be the optimal coupling of $w^{\bold v^i}$ and $w^{\mu_t,\widetilde r}$. For a couple of random variables $(V,\widetilde V)$ with distribution $p^i_{t,\tilde r, \bold v^i}$, define $P^{i,k}_{t,\tilde r, \bold v^i}$ to be the conditional law of $\widetilde V$, given that $V=V^{i+k}$, \textit{i.e.}:
\begin{equation}\label{definition P} P^{i,k}_{t,\tilde r, \bold v^i}(B) = p^i_{t,\tilde r, \bold v^i}\left(\{v^{i+k}\}\times B \Big | \{v^{i+k}\}\times \R^d\right) = \frac{1}{\gamma_k}p^i_{t,\tilde r, \bold v^i}(\{v^{i+k}\}\times B ),\end{equation}
for any Borel subset $B$ of $\R^d$. The second equality comes from the fact that $p^i_{t,\tilde r, \bold v^i}(\{v^{i+k}\}\times \R^d) = w^{\bold v^i}(\{v^{i+k}\})= \gamma_k$. Then there exists a function 
\begin{equation}\label{definition q} q^{i,k}_{t,\tilde r, \bold v^i}:\Lambda^{i+k} \to \R^d, \end{equation}
such that, if $r$ is uniformly distributed on $\Lambda^{i+k}$, $q^{i,k}_{t,\tilde r, \bold v^i}(r)$ has distribution $P^{i,k}_{t,\tilde r, \bold v^i}$. Defining $\Pi^i$ by
\begin{equation}\label{definition pi}\Pi^i_t(r,\tilde r, \bold v^i) = \sum_{k=-\ell N}^{\ell N} q^{i,k}_{t,\tilde r, \bold v^i}(r) \1_{r\in\Lambda^{i+k}},\end{equation}
the pair $(v^{[ Nr ]},\Pi^i_t(r,\tilde r, \bold v^i))$ has indeed distribution $p^i_{t,\tilde r, \bold v^i}$ if $r$ has distribution $\sigma^i(r)dr$. Indeed,
\begin{align*}
\Prb \left(v^{[ Nr ]} = V_t^{i+k}, \Pi^i_t\in B\right) &= \Prb \left( r\in\Lambda^{i+k}, \Pi^i_t\in B\right) \\
&= \Prb \left( r\in\Lambda^{i+k}\right)  \Prb \left(\Pi^i_t\in B | r\in\Lambda^{i+k} \right) \\
&= \gamma_k  \Prb \left(q^{i,k}_{t,\tilde r, \bold v^i}(r)\in B \right) \\
&= p^i_{t,\tilde r, \bold v^i}(\{V_t^{i+k}\}\times B).
\end{align*}
\eqref{coupling goal} is now obvious from the fact $(V_t^{[ Nr ]},\Pi_t(r))$ has distribution $p^i_{t,\tilde r, \bold v^i}$ when $r$ is distributed by $\sigma^i(r)dr$. To ensure that the construction of $\Pi^i$ detailed above can be done with the desired measurability properties, we follow \cite{Cortez}. Since the mapping $(t,\tilde r, \bold v^i) \mapsto (w^{\bold v^i}, w^{\mu_t,\widetilde r})$ is measurable, then we know from Corollary 5.22 in \cite{Villani2} that there exists a measurable mapping $(t,\tilde r, \bold v^i)\mapsto p^i_{t,\tilde r, \bold v^i}$ such that $p^i_{t,\tilde r, \bold v^i}(v,\widetilde v)$ is an optimal coupling between those two measures. $P^{i,k}_{t,\tilde r, \bold v^i}$ defined in \eqref{definition P} is a kernel from $\R_+\times B_j \times (\R^d)^{2\ell N +1}$ to $\R^d$, so by Lemma 2.22 in \cite{Kallenberg}, $q^{i,k}_{t,\tilde r, \bold v^i}$ in \eqref{definition q} can be defined as a mesurable function of $(t,r,\tilde r, \bold v^i)$. Finally the definition \eqref{definition pi} ensures that $\Pi^i$ is a mesurable function of $(t,r,\tilde r, \bold v^i)$. \\

\textit{2. Proof of \eqref{bound optimal coupling}}\\

By \eqref{coupling goal}, we can directly bound $\calW_1(w^{\bold V^i_t},w^{\mu_t,\widetilde r^i_0})$ to obtain a bound on \eqref{bound optimal coupling}. By Kantorovich-Rubinstein duality formula (see \cite{Villani2}), this term can be rewritten:
\begin{equation} \label{duality formula coupling bound}
\calW_1(w^{\bold V^i_t},w^{\mu_t,\widetilde r^i_0}) = \sup_{\substack{Lip(\varphi)\leq 1 \\\varphi(0)=0}} \Bigg | \sum_{k=-\ell N}^{\ell N} \gamma_k \varphi(V^{i+k}_t) - \int_{\T\times E} \Gamma_\ell(\widetilde r^i_0-r'') \varphi(v'') d\mu_t(r'',z'') \Bigg |.
\end{equation}
By the definition \eqref{definition gamma} of $\gamma_k$, we have
$$\left | \gamma_k - \frac{1}{N}\Gamma_\ell \left(\frac{k}{N}\right) \right | = \left| \int_{\Lambda^k}\left( \Gamma_\ell \left(r \right) - \Gamma_\ell \left(\frac{k}{N}\right) \right) dr\right | \leq \frac{K}{\ell^2} \int_{\Lambda^0} |r| dr \leq \frac{K'}{N^2\ell^2},$$
using the Lipschitz property for $\Gamma_\ell$ for the first inequality. As a consequence, we can introduce the empirical measure $\mu^N_t$ in \eqref{duality formula coupling bound} by replacing $\gamma_k$ and get the bound:
\begin{align*}
&\calW_1(w^{\bold V^i_t},w^{\mu_t,\widetilde r^i_0}) \\
  &\hspace{1cm}\leq \sup_{\substack{Lip(\varphi)\leq 1 \\\varphi(0)=0}} \Bigg | \int_{\T\times E} \varphi(v')  \Gamma_\ell\left(\frac{i}{N}-r' \right)d\mu^N_t(r',z') - \int_{\T\times E}\varphi(v'') \Gamma_\ell(\widetilde r^i_0-r'')d\mu_t(r'',z'') \Bigg | \\
           &\hspace{7cm} + \frac{K}{N\ell^2}\sup_{\substack{Lip(\varphi)\leq 1 \\\varphi(0)=0}}  \int_{\T\times E} |\varphi(v')| \1_{\{|i/N - r'|<\ell/2\}} d\mu^N_t(r',z').
\end{align*}
Therefore, bounding $\varphi(v')\leq |v'|$ in the last term on the one hand, and taking expectations and the mean over $1\leq i \leq N$ on the other hand, we get:
\begin{align}\label{bound optimal coupling interm 1}
\frac{1}{N}\sum_{i=1}^N \E\left[ \calW_1(w^{\bold V^i_t},w^{\mu_t,\widetilde r^i_0}) \right] &\leq \frac{1}{N}\sum_{i=1}^N \E \Bigg [\sup_{\substack{Lip(\varphi)\leq 1 \\\varphi(0)=0}} \Bigg | \int_{\T\times E} \varphi(v')  \Gamma_\ell\left(\frac{i}{N}-r' \right)d\mu^N_t(r',z') \nonumber \\
           &\quad\quad\quad\quad\quad\quad\quad\quad\quad\quad\quad\ - \int_{\T\times E}\varphi(v'') \Gamma_\ell(\widetilde r^i_0-r'')d\mu_t(r'',z'') \Bigg | \Bigg] \nonumber \\
        &\quad +\frac{K}{N\ell^2} \frac{1}{N}\sum_{i=1}^N \frac{1}{N} \sum_{k=-\ell N}^{\ell N} \E\left[ |V^{i+k}_t | \right].
\end{align}
The sum $1/N\sum_i \E[ |V^i_t |]$ is bounded by Lemma \ref{bound speed}, therefore the last term in \eqref{bound optimal coupling interm 1} is bounded by
\begin{equation}\label{bound optimal coupling interm 2} \frac{1}{N\ell^2} \frac{1}{N}\sum_{i=1}^N \frac{1}{N} \sum_{k=-\ell N}^{\ell N} \E\left[ |V^{i+k}_t | \right] \leq \frac{K}{N \ell}.\end{equation}
We can now focus on the first term of \eqref{bound optimal coupling interm 1}. Let $\pi_t$ be a coupling between $\mu^N_t$ and $\mu_t$. Introducing cross terms, we bound the supremum term in the expectation:
\begin{align}\label{bound optimal coupling interm 3}
&\sup_{\substack{Lip(\varphi)\leq 1 \\\varphi(0)=0}} \Bigg | \int_{\T\times E} \varphi(v')  \Gamma_\ell\left(\frac{i}{N}-r' \right)d\mu^N_t(r',z') - \int_{\T\times E}\varphi(v'') \Gamma_\ell(\widetilde r^i_0-r'')d\mu_t(r'',z'') \Bigg |  \nonumber \\
    &\hspace{0.3cm} \leq\sup_{\substack{Lip(\varphi)\leq 1 \\\varphi(0)=0}} \int_{(\T\times E)^2} | \varphi(v') - \varphi(v'') | \Gamma_\ell\left(\frac{i}{N}-r' \right) d\pi_t(y',y'') \nonumber \\
     &\hspace{4cm}+ \sup_{\substack{Lip(\varphi)\leq 1 \\\varphi(0)=0}} \int_{(\T\times E)^2}|\varphi(v'')| \left | \Gamma_\ell\left(\frac{i}{N}-r' \right) - \Gamma_\ell\left(\frac{i}{N}-r''\right)\right |d\pi_t(y',y'') \nonumber \\
     &\hspace{4cm}+ \sup_{\substack{Lip(\varphi)\leq 1 \\\varphi(0)=0}} \int_{(\T\times E)^2}|\varphi(v'')| \left | \Gamma_\ell\left(\frac{i}{N}-r'' \right) - \Gamma_\ell(\widetilde r^i_0-r'')\right |d\mu_t(r'',z'') \nonumber \\
&\hspace{0.3cm}\leq \int_{(\T\times E)^2} | v' - v'' | \Gamma_\ell\left(\frac{i}{N}-r' \right) d \pi_t(y',y'') +\frac{K}{\ell^2} \int_{(\T\times E)^2}|v''|\1_{\{ |\widetilde r_0^i -r''  |<\ell/2\}}  \left | r' - r''\right |d \pi_t(y',y'') \nonumber \\
      &\hspace{4cm}+\frac{K\epsilon_N}{\ell^2} \int_{(\T\times E)^2}\1_{\{ |\widetilde r_0^i -r''  |<\ell/2\}} |v''| d\mu_t(r'',z''),
\end{align}
where we used that $|\phi(v'')| \leq |v''|$ and the Lipschitz property for $\Gamma_\ell$ at the last line, as well as the fact that $i/N,\widetilde r^i_0 \in B_j$ for some $j$ and therefore their difference is smaller than $\epsilon_N$ for the last term (see \eqref{uniform bound axial parameters}). Using Lemma \ref{bound speed}, we can bound $\int_{\T\times E} |v''| d\mu_t$ by a constant. Combining \eqref{bound optimal coupling interm 2} and \eqref{bound optimal coupling interm 3} in \eqref{bound optimal coupling interm 1}, and taking into account the expectation and the mean over $i$, we finally get the bound
\begin{align*}
&\frac{1}{N}\sum_{i=1}^N \E\left[ \calW_1(w^{\bold V^i_t},w^{\mu_t,\widetilde r^i_0})\right] \leq K \E\left[ \int_{(\T\times E)^2} | v' - v'' | d \pi_t(y',y'') \right]  \\
    &\hspace{2cm}+ \frac{K}{\ell^2} \E\left[\int_{(\T\times E)^2}|v''| \1_{\{ |\widetilde r_0^i -r''  |<\ell/2\}} \left | r' - r''\right |d\pi_t(y',y'')\right] + \frac{K\epsilon_N}{\ell} + \frac{K}{N\ell},
\end{align*}
for all coupling $\pi_t$ of $\mu^N_t$ and $\mu_t$. As $\epsilon_N > 1/N$, \eqref{bound optimal coupling} follows.\\

\textit{3. Proof of \eqref{bound optimal coupling 2}}\\

Using the definition \eqref{definition pi} of $\Pi^i$ and the definition of $q^{i,k}$ we get
$$\int_{B_j}\left | \Pi^i_t(r,\tilde r^i_0, \bold V^i_{t}) \right | \sigma^i(r) dr = \int_{\R^d\times \R^d }|v''| \1_{\left \{v' \in \left \{ V^{k}_t, \, k\in N B_j \right \}\right \}} dp^i_{t,\tilde r_0^i, \bold V^i_t}(v',v'').$$
By Cauchy-Schwarz inequality,
\begin{align*}
\int_{B_j}\left | \Pi^i_t(r,\tilde r^i_0, \bold V^i_{t}) \right | \sigma^i(r) dr  &\leq \left(\int_{\R^d} \1_{\left \{v' \in \left \{ V^{k}_t, \, k\in N B_j \right \}\right \}} dw^{\bold V^i_t}(v')\right)^{1/2} \left(\int_{\R^d} |v''|^2 dw^{\mu_t,\widetilde r^i_0}(v'')\right)^{1/2} \\ 
&\leq \left(\sum_{\substack{k = -\ell N, \\ i+k \in NB_j}}^{\ell N}\gamma_k \right)^{1/2} \left(\int_{\T\times E} \Gamma_\ell\left(\widetilde r^i_0 - r''\right) |v''|^2 d\mu_t(r'',z'')\right)^{1/2} \\
&\leq K\frac{\epsilon_N^{1/2}}{\ell^{1/2}}\left(1+ \frac{1}{2}\int_{\T\times E} \Gamma_\ell\left(\widetilde r^i_0 - r\right) |v|^2 d\mu_t(r'',z'')\right),
\end{align*}
using at the last line that the coefficients $\gamma_k$ are bounded by $K/\ell$ and using the standard inequality $(1+x)^{1/2}\leq 1+x/2$ to bound the square root. Summing this last inequality over $i,j$, taking expectations and using that $\widetilde r^i_0$ is uniform random variable on $B_j$ we get:
$$\frac{1}{N} \sum_{j=1}^{\epsilon_N^{-1}} \sum_{i\in NB_j} \E\left[\int_{B_j}\left | \Pi^i_t(r) \right | \sigma^i(r) dr\right] \leq  K\frac{\epsilon_N^{1/2}}{\ell^{1/2}}\left(1+ \frac{1}{2}\int_{\T\times E} |v''|^2 d\mu_t(r'',z'')\right).$$
Applying Lemma \ref{bound speed} gives finally \eqref{bound optimal coupling 2}.

\end{proof}

\begin{proof}[Proof of Lemma \ref{Nonlinear process construction}]
Fix $1\leq j \leq \epsilon_N^{-1}$ and $i\in NB_j$. First, by the construction of $\Pi^i$ in Lemma \ref{Optimal Coupling solid}, $(t,\omega,r)\mapsto \Pi^i_t(r,\tilde r^i_0(\omega), \bold V^i_{t_-}(\omega))$ is measurable with respect to the product of the predictable sigma field in $(t,\omega)$ and the Borel sigma field of $\T$. Therefore the integral with respect to the Poisson random measure in \eqref{nonlinear process definition} is well-defined. Using the terminology of the next section, we are first going to check that the law of the process $\widetilde Y^i$ solves the \textit{linear} martingale problem \eqref{linear martingale problem} associated with $\mu$ and starting at $\mu^j_0$ and finish with a uniqueness argument to prove that its law is in fact $\mu^j$. For any $\psi\in C^1_b(\T\times E)$, by a direct computation from \eqref{nonlinear process definition}, we have
\begin{align}\label{application Poisson}
\psi(\widetilde Y^i_t) &= \psi(\widetilde Y^i_0) + \int_0^t \widetilde V^i_s\cdot \nabla_x\psi(\widetilde Y^i_s) ds \nonumber\\
     &\quad\quad - \int_0^t ds \left(\int_{\T\times E}\Phi_\ell(\widetilde r^i_s-r')\nabla W(\widetilde X^i_s-x') d\mu_s(r',z') + \nabla U(\widetilde X^i_s)\right)\cdot \nabla_v\psi(\widetilde Y^i_s)\nonumber \\
     &\quad\quad + \bar \gamma \int_0^tds \int_{\T}\left( \psi(\widetilde r^i_s,\widetilde X^i_s,\Pi^i_s(r')) - \psi(\widetilde r^i_s,\widetilde X^i_s,\widetilde V^i_s) \right) \sigma^i(r')dr' + M_t,
\end{align} 
where $M_t$ is the compensated martingale associated with the Poisson integral. By the definition of $\Pi^i$ in Lemma \ref{Optimal Coupling solid}, the last integral can be written
\begin{align*}
&\int_{\T}\left( \psi(\widetilde r^i_s,\widetilde X^i_s,\Pi^i_s(r')) - \psi(\widetilde r^i_s,\widetilde X^i_s,\widetilde V^i_s) \right) \sigma^i(r')dr' \\
   &\hspace{5cm}= \int_{\R^d}\left( \psi(\widetilde r^i_s,\widetilde X^i_s,v') - \psi(\widetilde Y^i_s) \right) dw^{\mu_s,\widetilde r^i_0}(v') \\
   &\hspace{5cm}= \int_{\T\times E} \Gamma_\ell(\tilde r^i_0 - r') \left( \psi(\widetilde r^i_s,\widetilde X^i_s,v') - \psi(\widetilde Y^i_s) \right) d\mu_s(r',x',v'),
\end{align*}
and we can therefore rewrite \eqref{application Poisson} as
$$\psi(\widetilde Y^i_t) = \psi(\widetilde Y^i_0) +\int_0^t \calL[\mu_s]\psi(\tilde Y_s) ds + M_t.$$
Hence the law of $\widetilde Y^i$ solves the linear martingale problem associated with $\mu$ and starting at $\mu^j_0$. We are now left with the task of proving that this property is sufficient to characterize the law $\mu^j_t(r,z) = \epsilon_N^{-1}\1_{r\in B_j}\mu_t(r,z)$.

To do so, we construct $\epsilon_N^{-1}$ processes $(\hat Y^j)_{1\leq j \leq \epsilon_N^{-1}}$, whose respective laws solve the linear martingale problem associated with $\mu$ and starting at $\mu^j_0$. Let $U$ be a uniform $\calF_0$-measurable random variable on $\T$, independent of any $\hat Y^j$. Now define the process $Y$ on $\T\times E$ by
$$Y = \sum_{j=1}^{\epsilon_N^{-1}} \hat Y^j \1_{U\in B_j}.$$
We can check that the law of $Y$ solves the linear martingale problem associated with $\mu$ and starting at $\mu_0$. This suffices to deduce that $Y$ has law $\mu$ by uniqueness of the solutions to the \textit{linear} martingale problem. In fact, it is easy to see that $Y_0$ has distribution $\mu_0$ and, for any $\psi\in C^1_b(\T\times E)$, one has
\begin{align*}
\psi(Y_t) = \sum_{j=1}^{\epsilon_N^{-1}} \psi(\hat Y^j_t)\1_{U\in B_j} &= \sum_{j=1}^{\epsilon_N^{-1}} \left(\psi(\hat Y^j_0) +  \int_0^t \calL[\mu_s] \psi( \hat Y^j_s) ds + M^{\psi,j}_t \right)\1_{U\in B_j} \\
&= \psi(Y_0) + \int_0^t  \calL[\mu_s] \psi(Y_s) ds + M^\psi_t,
\end{align*}
where $M^{\psi,j}$ are martingales and $M^\psi = \sum M^{\psi,j} \1_{U\in B_j}$ is therefore a martingale.

It now remains to check that $\hat Y^j$ has distribution $\mu^j$. For any measurable function $\varphi$ and for any $1\leq j \leq \epsilon_N^{-1}$, by independence of $U$ and $\hat Y^j$,
$$\E\left[\varphi(Y_t)\1_{r_t\in B_j}\right] = \E\left[\varphi(\hat Y^j_t)\1_{U\in B_j}\right] = \epsilon_N \E\left[\varphi(\hat Y^j_t)]\right].$$
Therefore $\hat Y^j$ has distribution $\mu^j$ as expected, since
$$\E\left[\varphi(\hat Y^j_t)\right] = \epsilon_N^{-1} \E\left[\varphi(Y_t)\1_{r_t\in B_j}\right] = \int_{\T\times E} \varphi d\mu^j_t,$$
and this concludes the proof.

\end{proof}

\begin{proof}[Proof of Lemma \ref{Initial positions}]
Recall that $r^i_0$ is uniformly distributed over $B_j$, where $j$ is such that $i\in NB_j$. We therefore already have the uniform bound
$$\left | \frac{i}{N} - \widetilde r^i_0 \right | \leq \epsilon_N.$$
Recall that $Z^i_0$ has distribution $f_0(i/N,z)dz$ and $\widetilde Z^i_0$ has distribution $\epsilon_N^{-1} \int_{B_j}f(r,z)drdz$, where the integral is only over $r$, \textit{i.e.} the distribution of $\widetilde Z^i_0$ is the marginal of $\mu^j_0$ on the $z$ coordinates. Therefore, we can choose $(Z^i_0,\widetilde Z^i_0)$ to be distributed as the optimal coupling between those two measures and get, by Kantorovich-Rubinstein duality:
\begin{align*}
\E\left[ | Z^i_0 - \widetilde Z^i_0 |\right] &= \sup_{\substack{Lip(\varphi)\leq 1 \\\varphi(0)=0}} \int_E \varphi(z) f_0\left(\frac{i}{N},z\right)dz - \epsilon_N^{-1} \int_{B_j\times E} \varphi(z) f_0(r,z) drdz \\
&= \epsilon_N^{-1}\sup_{\substack{Lip(\varphi)\leq 1 \\\varphi(0)=0}} \int_{B_j\times E} \varphi(z)\left( f_0\left(\frac{i}{N},z\right) - f_0(r,z) \right)drdz \\
&\leq \epsilon_N^{-1} \sup_{\substack{Lip(\varphi)\leq 1 \\\varphi(0)=0}} \int_{B_j\times E} |\varphi(z)| \left| \frac{i}{N} -r \right| h(z)drdz \\
&\leq \epsilon_N^{-1} \int_{B_j}\left| \frac{i}{N} -r \right| dr \int_E |z| h(z)dz \\
&\leq K \epsilon_N
\end{align*}
using hypothesis \ref{H5} at the third line and at the last line. This concludes the proof.
\end{proof}

\begin{proof}[Proof of Lemma \ref{bound speed}]
Let $Y_t$ be the canonical process associated with the solution to the nonlinear martingale problem \eqref{nonlinear mean field martingale problem} starting at $\mu_0$. Defining the energy
\begin{equation}\label{definition energy of one particle} \calE[\mu_t](Y_t) := \frac{1}{2} |V_t|^2 + \frac{1}{2}\int_{\T\times E}\Phi_\ell(r_0-r') W(X_t-x')d\mu_t(y') + U(X_t),\end{equation}
we can see that its expectation is constant, \textit{i.e.}
$$\E[\calE[\mu_t](Y_t)] = \int_{\T\times E} \calE[\mu_t](y) d\mu_t(y) = \int_{\T\times E} \calE[\mu_0](y) d\mu_0(y),$$ 
which is finite by assumption \ref{H7}. In particular, by assumption \ref{H1} on $U$, we get the moment bound we expected.\\

A similar argument holds for the particle system defined by \eqref{original system solid}, which clearly conserves energy in the sense that:
\begin{align*}
\int_{\T\times E} \calE[\mu^N_t](y) d\mu^N_t(y) &= \sum_{i=1}^N\left(\frac{1}{2}|V^i_t|^2 + \frac{1}{2}\sum_{k=-\ell N}^{\ell N} \phi_k W(X^i_t - X^{i+k}_t) + U(X^i_t) \right) \\
&= \int_{\T\times E} \calE[\mu^N_0](y) d\mu^N_0(y)
\end{align*}
almost surely for all $t>0$.
\end{proof}

\begin{proof}[Proof of Lemma \ref{Lemma bound potential energy}]
The proof of this lemma is similar to that of Lemma \ref{bound speed} and is again based on conservation of energy. Since the law of $\widetilde Y^i_t$ is given by $\mu^j_t$ if $i\in NB_j$, one has:
\begin{align*}
\frac{1}{N} \sum_{i=1}^N \E\bigg[ \int_{\T\times E}\frac{1}{\ell} \1_{|\widetilde r_0^i-r''|<\ell/2} &|\nabla W(\widetilde X^i_t -x'') |d\mu_t(r'',z'')\bigg] \\
&= \epsilon_N \sum_{j=1}^{\epsilon_N^{-1}} \int_{(\T\times E)^2}\frac{1}{\ell} \1_{|r'-r''|<\ell/2} |\nabla W(x' -x'') |d\mu^j_t(r',z') d\mu_t(r'',z'') \\
&= \int_{(\T\times E)^2} \frac{1}{\ell} \1_{|r'-r''|<\ell/2} |\nabla W(x' -x'')| d\mu_t(r',z') d\mu_t(r'',z'').
\end{align*}

By Jensen's inequality and assumption \ref{H1} on $W$, we get
\begin{align*}
\frac{1}{N} \sum_{i=1}^N \E\bigg[ \int_{\T\times E}\frac{1}{\ell} \1_{|\widetilde r_0^i-r''|<\ell/2} &|\nabla W(\widetilde X^i_t -x'') |d\mu_t(r'',z'')\bigg] \\
&\leq \left( c \int_{(\T\times E)^2} \frac{1}{\ell} \1_{|r'-r''|<\ell/2} W(x' -x'')d\mu_t(r',z') d\mu_t(r'',z'')\right)^{1/2}.
\end{align*}

It just remains to show that this last term is bounded by a constant that does not depend on $\ell$. This fact is true if we replace $\ell^{-1}\1_{|r'-r''|<\ell/2}$ by $\Phi_\ell(r'-r'')$, by conservation of energy, following the proof of Lemma \ref{bound speed}. However, $\Phi_\ell(r)$ is not necessarily bigger than $\ell^{-1}\1_{|r|<\ell/2}$, in particular for $r$ close to $\ell/2$ or $-\ell/2$. Instead, we are going to take advantage of the integration over $r'$ and $r''$ and of the translation invariance of the Lebesgue measure to properly achieve this bound.

Let $a$ be the only real number in $[0,1/2]$ such that $\phi(a)=\phi(-a)=\| \phi \|_\infty/2$. Adding several times shifted version of the function $\phi$, we can upper bound the indicator function. More precisely, there exists a finite integer $m$ such that
$$\frac{\| \phi \|_\infty}{2} \1_{|r|\leq 1/2} \leq \sum_{k=-m}^m \phi(r+ka).$$
From this, we deduce that
$$\frac{1}{\ell} \1_{|r|\leq \ell /2} \leq \frac{2}{\| \phi \|_\infty} \sum_{k=-m}^m \Phi_\ell(r+ka).$$
Inserting this in the integral term, it now suffices to bound it by
\begin{align*}
&\int_{(\T\times E)^2} \frac{1}{\ell}\1_{|r'-r''|<\ell/2} W(x' -x'')d\mu_t(r',z') d\mu_t(r'',z'') \\
&\hspace{5cm}\leq  \frac{2(2m+1)}{\| \phi \|_\infty}\int_{(\T\times E)^2} \Phi_\ell(r'-r'') W(x' -x'')d\mu_t(r',z') d\mu_t(r'',z''),
\end{align*}
where we used translation invariance of the Lebesgue measure on $\T$.

\end{proof}

\begin{proof}[Proof of Lemma \ref{lemma independent processes}]
Denote by $\bar \calM^i$ the point process on $\T\times \R^d$ with atoms given by $(t,\Pi_t(r))$ for any atom $(t,r)$ of $\calM^i$. Since the initial positions $(\bar Y^i_0)_{i\in NB_j}$ are independent, by the equation \eqref{definition independent nonlinear processes}, independence of the processes $(\bar Y^i)_{i\in NB_j}$ will follow from the independence of $(\bar \calM^i)_{i\in NB_j}$.

The independence of the point processes $(\calM^i)_{i\in NB_j}$ is straightforward from the fact they do not share atoms, but by the definition of $\Pi^i$ in Lemma \ref{Optimal Coupling solid}, this argument alone does not guarantee the independence of $(\bar \calM^i)_{i\in NB_j}$. Let $g_1,...,g_n$ be $n$ nonnegative compactly supported functions defined on $\R^d$ and $i_1,...,i_n$ be $n$ indices in $NB_j$. Let
$$G_t := \sum_{k=1}^n \int_0^t ds \int_{\R^d} g_k(v) d\bar \calM^{i_k}(s,v) = \sum_{k=1}^n \int_0^t ds \int_{\T} g_k(\Pi^{i_k}_s(r)) d \calM^{i_k}(s,r),$$
by definition of $\bar \calM^{i_k}$. Since $(\calM^{i_k})_{1\leq k \leq n}$ do not share atoms almost surely, then with probability one,
$$\e^{-G_t} = 1 + \sum_{k=1}^n \int_0^t ds \int_{\T} \left(\exp\left(-G_{s_-} - g_k(\Pi^{i_k}_s(r))\right) - \exp\left(-G_{s_-} \right) \right) d \calM^{i_k}(s,r).$$
Taking expectations, we get
$$\E\left[ \e^{-G_t} \right] = 1 + \int_0^t ds \E\left[ \e^{-G_s} \right] \left(\sum_{k=1}^n \int_{\R^d} \left(\e^{ - g_k(v')} - 1 \right) dw^{\mu_s,\bar r_0^{i_k}}(v') \right),$$
which gives
$$\E\left[ \e^{-G_t} \right]  = \prod_{k=1}^n \exp\left( \int_0^t ds \int_{\R^d} \left(\e^{ - g_k(v')} - 1 \right) dw^{\mu_s,\bar r_0^{i_k}}(v')\right),$$
and this concludes the independence of $(\bar \calM^i)_{i\in NB_j}$.

\end{proof}


\section{The nonlinear martingale problem}\label{Section 2}




This section is devoted to the proof of Proposition \ref{nonlinear mean field martingale problem proposition}. Recall the formalism of \eqref{original system solid new formalism} with a point process that selects directly the neighbours' velocities. We expect that, as $N$ goes to infinity, the typical trajectories of the particles in the system \eqref{original system solid} should be close to the solution of the nonlinear stochastic differential equation
\begin{equation}\label{nonlinear mean field SDE}
\left\{
\begin{array}{l}
  dX_t = V_t dt \\
  dV_t = -\left(\int_{\T\times E} \Phi_\ell(r_0-r')\nabla W(X_t-x')d\mu_t(r',z') + \nabla U(X_t)\right)dt \\
      \hspace{9cm}+ \int_{\R^d} (v' - V_{t^-})d\calN^{\mu,r_0}(t,v'),
\end{array}
\right.
\end{equation}
where $\calN^{\mu,r_0}$ is a Poisson random measure on $\R_+\times \R^d$, independent of $(X_0,V_0)$ given $r_0$, with intensity
\begin{equation}\label{intensity} \bar\gamma dt dw^{\mu_t,r_0}(v'),\end{equation}
given $r_0$. $w^{\mu_t,r_0}$ is the velocity marginal of $\mu_t$, weighted by the function $\Gamma_\ell$ around $r_0$, as defined in \eqref{weighted distribution} (replacing obviously $\tilde r_0^i$ by $r_0$ in this formula). Moreover, we set $r_t= r_0$ for all time $t\geq 0$. Formally, \eqref{nonlinear mean field SDE} is given by taking the limit $\mu$ of $\mu^N$ in \eqref{original system solid new formalism}.

To deal with the nonlinear problems, we will need to work with solutions of the associated \textit{linear} problems. Let $Q\in\calM^1(\calD)$ be a fixed probability measure. We say that $\mu$ is a solution to the \textit{linear martingale problem} associated with $Q$ and starting at $\nu_0\in\calM^1(\T\times E)$ if $\mu_0=\nu_0$ and
\begin{equation}\label{linear martingale problem}
M^{\psi}_t = \psi(Y_t) - \psi(Y_0) - \int_0^t \mathcal{L}[Q_s] \psi(Y_s)ds
\end{equation}
is a martingale under $\mu$ for any $\psi\in C^1_b(\T\times E)$, where $\mathcal{L}$ is the operator defined by \eqref{definition operator L}. Similarly, the \textit{linear mean field SDE} associated with $Q$ will refer to the SDE \eqref{nonlinear mean field SDE} where $\mu$ is replaced by $Q$. 

\begin{Prop}\label{linear mean field martingale problem proposition}
Let $\nu_0 \in \calM^1(\T\times E)$ be such that its $r$-marginal is the Lebesgue measure on $\T$ and $\int_{\T\times E} |z| d\nu_0$. Let $Y_0$ be random variable with law $\nu_0$ and let $Q\in\calD$ such that for all $T\geq 0$,
\begin{equation}\label{h1 for linear problem} \sup_{t\leq T} \int_{\T\times E} |z| dQ_t(r,z) < \infty.\end{equation}
There is a unique solution to the linear stochastic differential equation, both pathwise and in law. Its law $\mu$ is the unique solution to the linear martingale problem \eqref{linear martingale problem} associated with $Q$ starting at $\nu_0$. Moreover, for any $T>0$, $\mu$ satisfies
\begin{equation}\label{h1 for solution} \sup_{t\leq T} \int_{\T\times E} |z| d\mu_t(r,z) <\infty,\end{equation}
and for any $t>0$, and any measurable $\varphi:\T\to \R$
\begin{equation}\label{h2 for solution} \int_{\T\times E}\varphi(r)d\mu_t(r,z) =  \int_{\T\times E}\varphi(r)d\nu_0(r,z). \end{equation}
\end{Prop}

The proof of Proposition \ref{linear mean field martingale problem proposition} is classical and left to the reader. We only give the main steps for sake of completeness. It suffices to prove first a contraction result for the solutions of the linear SDE, which is straightforward in the linear case. Then the law of the solutions of the linear SDE solve the linear martingale problem \eqref{linear martingale problem}. To prove that it is the only one, we can follow \cite{Desvillettes,FournierMeleard} and adapt the classical representation arguments from \cite{ElKaroui}, \cite{Lepeltier} or \cite{Tanaka} by proving that the canonical process associated with a solution to the martingale problem can be represented as a solution to the linear SDE on an enlarged probability space.

With this result, we are now able to prove Proposition \ref{nonlinear mean field martingale problem proposition}. The proof is inspired by the proof of Graham \cite{Graham} and relies on a contraction estimate in the space $\calM^1(\calD_T)$ for some $T>0$, where $\calD_T := D([0,T],\T\times E)$ is the set of right-continuous functions with left-limits defined on the time interval $[0,T]$ with values in $\T\times E$. To take into account the particular role of the parameter $r$, we will use a new appropriate Wasserstein-type distance. Let us first introduce the more classical distance that we will use throughout the proof. From the uniform distance
\begin{equation}\label{definition uniform distance}\| y^1-y^2 \|_{\infty,T} := \sup_{t\leq T}|y^1_t-y^2_t| \end{equation}
between $y^1,y^2 \in\calD_T$ and from the Skorokhod distance $\|\cdot\|_{0,T}$ (see \cite{Billingsley}), we can consider the associated Wasserstein distances between $Q^1,Q^2 \in \calM^1(\mathcal D_T)$:
$$\calW_{\infty,T} (Q^1,Q^2) = \inf_{\pi\in\Pi(Q^1,Q^2)} \int \| y-y' \|_{\infty,T} d\pi(y,y'),$$
and similarly for $\calW_{0,T}$. In particular
\begin{equation}\label{trivial bound wasserstein}\calW_{0,T} (Q^1,Q^2) \leq \calW_{\infty,T} (Q^1,Q^2),\end{equation}
holds for any measures $Q^1$ and $Q^2$ by comparison of the two distances $\|\cdot\|_{\infty,T}$ and $\|\cdot\|_{0,T}$. We will also use the notation $\calW_{\infty,T}$ for Wasserstein distances between probability measures on the space of trajectories on $\T\times E$, $E$ and $\R^d$. We will only deal with measures $Q^1,Q^2 \in \calM^1(\mathcal D)$ satisfying the following properties
\begin{enumerate}[label=\textbf{(h\arabic*)}]
\item \label{h1} $\sup_{t\leq T} \int_{\T\times E} |z| dQ^i_t <\infty$ for $i=1,2$ and for all $T>0$.
\item \label{h2} $Q^1$ and $Q^2$ only charge trajectories $(r_t,x_t,v_t)_{t\geq 0}$ with constant $r_t = r_0$. Moreover, the $r$-marginal of $Q^1_0$ and $Q^2_0$ is the Lebesgue measure on $\T$.
\end{enumerate}
Under \ref{h2}, define the probability measures $Q^{1,r}$ and $Q^{2,r}$ on $D(\R_+,E)$ to be the conditional distributions of $Q^1$ and $Q^2$ on the $z$-trajectories given $r$. 
\begin{Def}[Sliced Wasserstein distance]\label{sliced wasserstein}
The sliced Wasserstein distance between two measures $Q^1,Q^2\in \calM^1(\mathcal D_T)$ satisfying \ref{h2} is the distance
\begin{equation}\label{new distance definition}  \calS \calW_{\infty,T}(Q^1,Q^2) := \int_\T dr \calW_{\infty,T}(Q^{1,r},Q^{2,r}). \end{equation}
\end{Def}

Such a terminology has been used in a different context (see Chapter 5.5.4 in \cite{Santambrogio}), although the ideas behind both definitions are close. Therefore we will stick to this terminology. We will actually prove the contraction estimate for sliced Wasserstein distance by constructing an explicit coupling of the solutions of the linear martingale problem associated with $Q^1$ and $Q^2$. Moreover, the bound
$$\calW_{\infty,T}(Q^1,Q^2) \leq \calS \calW_{\infty,T}(Q^1,Q^2).$$
follows directly by considering the coupling $(r,Z,r,Z')\in (\mathcal D_T)^2$ of $Q^1$ and $Q^2$, where $r$ is a uniform random variable on $\T$ and $(Z,Z')$ has distribution given by the optimal coupling of $Q^{1,r}$ and $Q^{2,r}$. Combined with \eqref{trivial bound wasserstein}, we get
\begin{equation}\label{trivial bound wasserstein 2}\calW_{0,T} (Q^1,Q^2) \leq \calS\calW_{\infty,T} (Q^1,Q^2).\end{equation}

\begin{proof}[Proof of Proposition \ref{nonlinear mean field martingale problem proposition}]

In the whole proof, $K$ and $K'$ are positive constants that change from one line to another and depend on the parameters of the problem $W$, $U$, $\phi$ $\gamma$ but also $\ell$ and $\bar\gamma$, since we do not need to keep track of these elements anymore.\\

\textit{1. Contraction estimate for the sliced Wasserstein distance on a short time interval} \\

Let $Q^1$ and $Q^2$ be two elements of $\calM^1(\calD)$ satisfying \ref{h1} and \ref{h2}. We construct two coupled solutions $Y^1=(r^1,Z^1)$ and $Y^2=(r^2,Z^2)$ of the \textit{linear} SDE associated with $Q^1$ and $Q^2$ respectively and with initial distribution $\nu_0$. We first couple their initial positions by setting $r^1_0=r^2_0 = r_0$, where $r_0$ is a uniform random variable on $\T$, and we set $Z^1_0=Z^2_0$, with $Z^1_0$ distributed with the probability measure $\nu^{r_0}_0$. Denoting by $\mu^i$ the law of $Y^i$, $i,=1,2$, we immediately get
\begin{equation}\label{bound bar W}
\calS \calW_{\infty,T}(\mu^1,\mu^2) \leq \E\left[\| Z^1 - Z^2 \|_{\infty,T} \right].
\end{equation}

We then construct the Poisson random measures $\calN^{Q^1,r_0}$, $\calN^{Q^1,r_0}$ that drive the two processes. By \ref{h2}, the time intensities of the Poisson random measures $\calN^{Q^1,r_0}$ and $\calN^{Q^1,r_0}$ are equal (to $\bar\gamma$), thus we can take the same jump times $(T_n)_{n\geq 0}$ for both. We will take $V^1_{T_n}$ and $V^2_{T_n}$ to be distributed as the optimal coupling of $w^{Q^1_{T_n},r_0}$ and $w^{Q^2_{T_n},r_0}$, that is to say:
\begin{equation}\label{optimal coupling weighted measure} \E \left[| V^1_{T_n} - V^2_{T_n}| | r_0,T_n\right] = \calW_1(w^{Q^1_{T_n},r_0},w^{Q^2_{T_n},r_0}).\end{equation}

We have a good control on \eqref{optimal coupling weighted measure} with the sliced distance by the following argument. Let $\pi^{r} \in \Pi(Q^{1,r},Q^{2,r})$ be the optimal coupling between $Q^{1,r}$ and $Q^{2,r}$ for $\calW_{\infty,T}$, for $r\in \T$. Using Kantorovich-Rubinstein duality formula and hypothesis \ref{h2}, we can bound $\mathcal{W}_1(w^{Q^1_{t},r_0},w^{Q^2_{t},r_0})$ for all $t\geq 0$:
\begin{align} \label{control jump times}
\mathcal{W}_1(w^{Q^1_t,r_0},w^{Q^2_t,r_0}) &= \sup_{Lip(\varphi)\leq 1}  \left |\int_{\T} dr \Gamma_\ell(r_0-r)\int_{E^2}(\varphi(v')  - \varphi(v'')) d\pi^{r}_t(z',z'') \right | \nonumber \\
&\leq K \calS \calW_{\infty,T}(Q^1,Q^2),
\end{align}
bounding $\Gamma_\ell$ by a constant and the difference of the Lipschitz functions by $|v'-v''|$. As a consequence we control what happens at jump times. We will only be concerned by the jump times before $T$, so set $S_n = T\wedge T_n$. The quantity $d_n :=  \|Z^1_t -Z^2_t\|_{\infty,S_n}$ is clearly bounded by
$$d_n \leq \sup_{t<S_n} |Z^1_t - Z^2_t | + |Z^1_{S_n} - Z^2_{S_n}|.$$

The last term can be bounded in expectation using that $X$ is time-continuous and $V$ is only discontinuous at jump times, in which case \eqref{control jump times} controls the difference of the two velocities:
$$ \E\left[|Z^1_{S_n} - Z^2_{S_n}|\right] \leq \E\left[\sup_{t<S_n} | Z^1_t-Z^2_t|\right]  +  K \calS \calW_{\infty,T}(Q^1,Q^2).$$
Combining this with the previous inequality gives:
\begin{equation}\label{recursion for dn}
\E[d_n] \leq K' \E[d_{n-1}] + K' \E\left[\sup_{S_{n-1}<t<S_n}  | Z^1_t-Z^2_t| \right]+ K' \calS \calW_{\infty,T}(Q^{1},Q^{2}).\end{equation}

It just remains to control properly the second term of \eqref{recursion for dn} to get a recursive inequality on $\E[d_n]$. This second term corresponds to a deterministic evolution between two jump times, so we need in particular to control the difference between the force terms at time $t$. By the Lipschitz property for $\nabla W$ and \ref{h2}, we have
\begin{align*} 
&\Bigg | \int_{\T\times E} \Phi_\ell(r_0-r')\nabla W(X^1_t-x')dQ^1_t(r',z') - \int_{\T\times E} \Phi_\ell(r_0-r'')\nabla W(X^2_t-x'')dQ^2_t(r'',z'')\Bigg |  \\
    &\hspace{4cm}\leq  \int_{\T} dr \Phi_\ell(r_0-r)\int_{E^2} \left | \nabla W(X^1_t-x') - \nabla W(X^2_t-x'') \right | d\pi^{r}_t(z',z'') \\
    &\hspace{4cm}\leq K | X^1_t -X^2_t | +  K \calS \calW_{\infty,T}(Q^{1},Q^{2}),
\end{align*}
introducing cross terms at the last line. The contribution of the difference between the pinning terms $\nabla U$ gives an extra $ K |X^1_t - X^2_t|$ term. Hence, we get for $S_{n-1}<t<S_n$
\begin{align*}
| Z^1_t - Z^2_t | &\leq |Z^1_{S_{n-1}} - Z^2_{S_{n-1}} | + K \int_{S_{n-1}}^t |Z_s^1 - Z_s^2|ds  + K T \calS\calW_{\infty,T}(Q^{1},Q^{2})\\
&\leq d_{n-1} + K T \sup_{S_{n-1}<s<S_n}|Z^1_s - Z^2_s| + K T\calS \calW_{\infty,T}(Q^{1},Q^{2}).
\end{align*}

Taking the supremum for $S_{n-1}<t<S_n$ on the left-hand side and choosing $T$ sufficiently small so that $KT<1$, we get the following bound:
\begin{equation}\label{hamiltonian dynamic bound}
\sup_{S_{n-1}<t<S_n}|Z^1_t - Z^2_t| \leq \frac{1}{1-K T} d_{n-1} + \frac{KT}{1-KT}\calS\calW_{\infty,T}(Q^{1},Q^{2}). 
\end{equation}

We can now combine this bound \eqref{hamiltonian dynamic bound} for the deterministic dynamic evolution to \eqref{recursion for dn} to get that:
\begin{align*}
\E[d_n] &\leq K\frac{2-KT}{1-KT} \E[d_{n-1}] + \frac{K}{1-KT} \calS\calW_{\infty,T}(Q^{1},Q^{2}) \\
&\leq a_T \E[d_{n-1}] + b_T \calS\calW_{\infty,T}(Q^{1},Q^{2}),
\end{align*}
where we denoted for simplicity
$$a_T = K\frac{2-KT}{1-KT}, \hspace{0.5cm} b_T = \frac{K}{1-KT}.$$

By recursion and using the initial condition $d_0=0$, we get
\begin{equation} \label{control distance recursion} \E[d_n] \leq  \frac{a_T^n-1}{a_T-1}b_T\calS \calW_{\infty,T}(Q^1,Q^2). \end{equation}

Now that we control the trajectories until time $S_n$, we can extend this control up to time $T$. Let $N_T$ be the number of jumps during the interval of time $[0,T]$. Then
\begin{equation}\label{last control hamiltonian}\E\left[ \| Z^1 - Z^2 \|_{\infty,T} \right] \leq \E[d_{N_T}] + \E\left[\sup_{T_{N_T}< t\leq T} |Z^1_t -Z^2_t|\right].\end{equation}

As there is no jump between times $T_{N_T}$ and $T$, we can apply the same estimates that lead to \eqref{hamiltonian dynamic bound} to the second term on the right-hand side of \eqref{last control hamiltonian}:
\begin{align*}
\E\left[\sup_{T_{N_T}< t\leq T} |Z^1_t -Z^2_t|\right] &\leq \frac{1}{1-K T} \E[d_{N_T}] + \frac{KT}{1-KT}\calS\calW_{\infty,T}(Q^1,Q^2)\\
&\leq \left(\frac{a_T}{K}-1\right) \E[d_{N_T}] + \left(\frac{b_T}{K}-1\right)\calS \calW_{\infty,T}(Q^1,Q^2),
\end{align*}
in terms of $a_T$ and $b_T$. All in all, combining this last inequality with \eqref{control distance recursion} and \eqref{last control hamiltonian} leads to
\begin{align} \label{control coupling by sliced wasserstein}
\E\left[ \|Z^1_t -Z^2_t\|_{\infty,T}\right] &\leq \frac{a_T}{K} \E[d_{N_T}] + \left(\frac{b_T}{K}-1\right) \calS \calW_{\infty,T}(Q^1,Q^2)\nonumber \\
& \leq \left(\frac{a_Tb_T}{K(a_T-1)} \left(\E[a_T^{N_T}]-1\right) + \frac{b_T}{K} -1 \right) \calS \calW_{\infty,T}(Q^1,Q^2).
\end{align}
As $N_T$ has a Poisson distribution with parameter $\bar T$ given $r_0$, the coefficient multiplying $\calS \calW_{\infty,T}(Q^1,Q^2)$ can be computed explicitly:
\begin{align*}
c_T &:= \frac{a_Tb_T}{K(a_T-1)} \left(\E[a_T^{N_T}]-1\right) + \frac{b_T}{K} -1\\
&= \frac{a_Tb_T}{K(a_T-1)} \left(\exp\left(\bar\gamma T(a_T-1)\right)-1\right) + \frac{b_T}{K}-1.
\end{align*}
Since $b_0=K$, it is easy to see that $c_0 = 0$. As $c_T$ is a continuous function of $T$, we can choose $T$ small enough so that $c_T<1$. Combining \eqref{control coupling by sliced wasserstein} with \eqref{bound bar W} finally gives the following contraction estimate for a short time $T>0$:
\begin{equation}\label{bound3 solid}
\calS \calW_{\infty,T}(\mu^1,\mu^2) \leq c_T \calS \calW_{\infty,T}(Q^1,Q^2).
\end{equation}
\\

\textit{2. Existence and uniqueness of solutions to the martingale problem} \\

Uniqueness in $\mathcal{M}^1(\mathcal{D}_T)$ follows from the contraction estimate \eqref{bound3 solid}. Existence in $\mathcal{M}^1(\mathcal{D}_T)$ follows from a classical iteration procedure. By \eqref{h1 for solution} and \eqref{h2 for solution} in Proposition \ref{linear mean field martingale problem proposition}, the solutions of the linear martingale problem are well defined iteratively and satisfy hypotheses \ref{h1} and \ref{h2}. By the contraction estimate \eqref{bound3 solid}, we conclude that the sequence is Cauchy for $\calS\calW_{\infty,T}$ and therefore for $\calW_{0,T}$ by \eqref{trivial bound wasserstein 2}. From the completeness of $(\mathcal D_T, \| \cdot \|_{0,T})$, $\calW_{0,T}$ makes $\calM^1(\mathcal D_T)$ a complete metric space (see \cite{Villani2}). Consequently, the sequence converges in $\calM^1(\mathcal D_T)$.

Finally, we extend existence and uniqueness to $\mathcal{M}^1(\mathcal{D})$ by the following argument. Let $\mu$ be the solution to the nonlinear martingale problem starting at $\nu_0$ on $\mathcal{M}^1(\mathcal{D}_T)$ and let $Q^1$ and $Q^2$ be two measures on $\mathcal{M}^1(\mathcal{D})$ such that $Q^1_t = Q^2_t = \mu_t$ for $t\leq T$. Define $\mu^1$ and $\mu^2$ to be the solutions of the linear martingale problem associated with $Q^1$ and $Q^2$ on $\mathcal{M}^1(\mathcal{D}_{2T})$. Thus $\mu^1_T = \mu^2_T$ by uniqueness of the linear martingale problem on $\mathcal{M}^1(\mathcal{D}_{T})$ and the same contraction estimate as \eqref{bound3 solid} hold for the interval $[T,2T]$ between $\mu^1,\mu^2$ and $Q^1,Q^2$. Thus existence and uniqueness extends to $\mathcal{M}^1(\mathcal{D}_{2T})$ and by immediate recurrence to $\mathcal{M}^1(\mathcal{D})$.

\end{proof}

The same Picard iteration used in the previous proof enables to prove the following Proposition.

\begin{Prop}
There is existence and uniqueness in law of solutions to the SDE \eqref{nonlinear mean field SDE}.
\end{Prop}


\section{Energy transport}\label{Section 4}

We finish by studying the transport of energy in our model for a class of pinning potentials $U$ and proving Proposition \ref{Energy evolution}. The proof is in two parts. In the first one, we prove a diffusion equation for the energy associated with the limit measure $\mu_t$ at the timescale $t\ell^{-2}$. In the second part, we prove that the energy of the particle system converges to the energy associated with $\mu_t$ in that timescale.

\subsection{Diffusion equation for the limit measure}

First, let us start by proving that the symmetry \ref{H8} is preserved at any time $t>0$ for symmetric potentials.

\begin{Lem}\label{Symmetry}
Under \ref{H8}, if $U$ and $W$ are symmetric, then the symmetry $\mu_t(r,z) = \mu_t(r,-z)$ for any $(r,z)\in\T\times E$ is preserved at any later time $t>0$.
\end{Lem}

\begin{proof}
Let $(Y_t)_{t\geq 0}=(r_t, X_t,V_t)_{t\geq 0}$ be the canonical process associated with the solution $\mu$ to the nonlinear martingale problem \eqref{nonlinear mean field martingale problem}. Define  $\mu^*(r,z) = \mu(r,-z)$, $\psi^*(r,z) = \psi(r,-z)$ and $Y_t^* = (r_t,-Z_t)$ for simplicity. Then, we can rewrite
\begin{align*}
&\calA[\mu_s]\psi(Y_s) \\
    &\quad= V_s\cdot \nabla_v\psi(Y_s) -\left( \nabla U(X_s)+ \int_{\T\times E} \Phi_\ell(r_s-r')\nabla W(X_s-x') d\mu_s(r',z') \right) \cdot \nabla_x\psi(Y_s) \\
    &\quad= -V_s\cdot \nabla_v\psi^*(Y_s^*) +\left( \nabla U(X_s)+ \int_{\T\times E} \Phi_\ell(r_s-r')\nabla W(X_s+x') d\mu_s^*(r',z') \right) \cdot \nabla_x\psi^*(Y_s^*) \\
    &\quad= \calA[\mu^*_s]\psi^*(Y^*_s),
\end{align*}
using antisymmetry of $\nabla U$ and $\nabla W$ at the last line. Similarly,
$$\calS[\mu_s]\psi(Y_s) = \int_{\T\times E} \Gamma_\ell(r_s-r') \left(\psi^*(r_s,-X_s,v') - \psi^*(Y^*_s)\right) d\mu^*_s(r',z') = \calS[\mu^*_s]\psi^*(Y^*_s).$$
As a consequence, for any $\psi \in  C_b^1(\T\times E)$,
\begin{align*}
\psi^*(Y^*_t) -\psi^*(Y^*_0) + \int_0^t \calL[\mu^*_s]\psi^*(Y^*_s) ds &= \psi(Y_t) -\psi(Y_0) + \int_0^t \calL[\mu_s]\psi(Y_s) ds \\
&= M^\psi_t,
\end{align*}
where $M^\psi_t$ is a martingale. Since $Y^*$ has distribution $\mu^*$, $\mu^*$ is also solution to the nonlinear martingale problem \eqref{nonlinear mean field martingale problem} starting at $\mu_0^*=\mu_0$. By uniqueness, $\mu^* = \mu$, which implies the property we were looking for.
\end{proof}

As a consequence of this symmetry property, the integral $\int_{\R^d} x'd\mu_t(r',z')$ is null. In particular, if the interaction potential is harmonic $W(x) = |x|^2/2$, then the Vlasov equation \eqref{Vlasov + noise} is reduced to the simplified version \eqref{simple Vlasov}. We will always consider harmonic interaction potentials from now on. Let $Y$ be a solution to the nonlinear SDE \eqref{nonlinear mean field SDE} with spatial parameter $r\in\T$ and recall the notation from \eqref{definition energy of one particle}:
\begin{align}\label{definition energy of one particle 2}
\calE[\mu_t](Y_t) &= \frac{1}{2}|V_t|^2 + \frac{1}{2}\int_{\T\times E} \Phi_\ell(r-r') W(X_t-x')d\mu_t(r',z') + U(X_t) \nonumber \\
&= \frac{1}{2}|V_t|^2 + \frac{1}{4}|X_t|^2 + \frac{1}{4}\int_{\T\times E} \Phi_\ell(r-r')|x'|^2d\mu_t(r',z') + U(X_t),
\end{align}
since the term $\frac{1}{2}X_t \cdot \int_{\T\times E} \Phi_\ell(r-r')x'd\mu_t(r',z') =0$ by Lemma \ref{Symmetry}. Define also the local energy associated with the distribution $\mu$ of $Y$, given $r_0=r$, by
\begin{align}\label{definition energy mu}
\calE_t(r) &= \E\Big [\calE[\mu_t](Y_t) | r_0=r \Big] \nonumber \\
&= \int_E \left( \frac{1}{2}|v|^2 + \frac{1}{2}\int_{\T\times E} \Phi_\ell(r-r') W(x-x')d\mu_t(r',z') + U(x) \right) d\mu_t(r,x,v) \nonumber \\
&= \int_E \left( \frac{1}{2}|v|^2 + \frac{1}{4}\int_{\T\times E} \Phi_\ell(r-r') |x'|^2d\mu_t(r',z')  + \frac{1}{4} |x|^2 + U(x) \right) d\mu_t(r,x,v).
\end{align}

When $W$ is harmonic, the contribution of the neighboring parts near the axial position $r$ in \eqref{definition energy mu} is only given by the integral $\int_{\T\times E} \Phi_\ell(r-r') |x'|^2d\mu_t(r',z')$, which does not involve the variable $x$. Let us investigate in more detail the evolution of $\calE_t(r)$. By a straightforward calculation from \eqref{nonlinear mean field SDE}, we get that for any $g\in \calC^2(\T)$,
\begin{align} \label{diffusion energy limit 1}
&\int_{\T} \calE_{t\ell^{-2}}(r) g(r)dr - \int_{\T} \calE_0(r) g(r)dr  \nonumber\\
      &\hspace{1cm}=\E\Big [\calE[\mu_{t\ell^{-2}}](Y_{t\ell^{-2}}) g(r) \Big] - \E\Big [\calE[\mu_0](Y_0) g(r) \Big] \nonumber \\
      &\hspace{1cm}= - \int_0^{t\ell^{-2}} ds \int_{\T^2} g(r) \Phi_\ell(u)j_s^{r,r+u;a} drdu- \bar\gamma \int_0^{t\ell^{-2}} ds \int_{\T^2} g(r) \Gamma_\ell(u)j^{r,r+u;s}_s drdu,
\end{align}
where
\begin{align*}
j^{r,r+u;a}_s &= \frac{1}{2} \int_{E^2} (v + v')\cdot\nabla W(x-x') d\mu_s(r+u,z') d\mu_s(r,z) \\
&= \frac{1}{2} \int_{E} v \cdot x d\mu_s(r,z) - \frac{1}{2} \int_{E} v' \cdot x' d\mu_s(r+u,z') 
\end{align*}
is the mechanical contribution to the energy current and
$$j^{r,r+u;s}_s = \int_{E} \frac{1}{2}|v|^2 d\mu_s(r,z) - \int_{E} \frac{1}{2}|v|^2d\mu_s(r+u,z').$$
is the stochastic contribution to the energy current. Rearranging terms by an integration by parts in \eqref{diffusion energy limit 1}, since $r\in\T$ and $\Phi_\ell$ and $\Gamma_\ell$ are symmetric, we get
\begin{align}\label{diffusion energy limit 2}
&\int_{\T} \calE_{t\ell^{-2}}(r) g(r)dr - \int_{\T} \calE_0(r) g(r)dr \nonumber \\
      &\hspace{0.1cm}= \int_0^{t\ell^{-2}} ds \int_{\T\times E} \frac{1}{2} v\cdot x d\mu_s(r,z) \int_\T \Phi_\ell(u) \left(g(r+u) - g(r)\right) du \nonumber \\
             &\hspace{4cm}+ \bar\gamma \int_0^{t\ell^{-2}} ds  \int_{\T\times E} \frac{1}{2} |v|^2 d\mu_s(r,z)\int_{\T} \Gamma_\ell(u)\left(g(r+u)-  g(r)\right) du \nonumber \\
      &\hspace{0.1cm}= c_\phi \int_0^t ds \int_{\T\times E} \frac{1}{2}v\cdot x  g''(r) d\mu_{s\ell^{-2}}(r,z) + \bar\gamma c_\gamma \int_0^t ds \int_{\T\times E} \frac{1}{2}|v|^2  g''(r) d\mu_{s\ell^{-2}}(r,z) + O(t\ell),             
\end{align}
where the error term comes from a Taylor expansion of $g$ at the last line and we usied the uniform in time moment bound of Lemma \ref{bound speed}. The constant $c_\phi$ is equal to
$$c_\phi = \frac{1}{2}\int_{-1/2}^{1/2} u^2\phi(u) du,$$
and similarly for $c_\gamma$. Notice that, even if we have the symmetry property given by Lemma \ref{Symmetry}, this is not sufficient to prove that the hamiltonian contribution $\int x\cdot vd\mu_t(r,z)$ in \eqref{diffusion energy limit 2} vanishes. However, it is easy to prove from \eqref{nonlinear mean field SDE} that
\begin{align}\label{derivative x square}
&\int_{\T\times E} |x|^2 g''(r)d\mu_{t\ell^{-2}}(r,z) - \int_{\T\times E} |x|^2 g''(r)d\mu_0(r,z) \nonumber \\
     &\hspace{7cm}= \ell^{-2} \int_0^t ds \int_{\T\times E} \frac{1}{2} x\cdot v g''(r)d\mu_{s\ell^{-2}}(r,z),
\end{align}
and consequently, applying uniform in time moment bound of Lemma \ref{bound speed} to the left-hand side of \eqref{derivative x square}, we get an estimate of the right-hand side of \eqref{derivative x square}:
\begin{equation}\label{vanishing hamiltonian current} \int_0^t ds \int_{\T\times E} \frac{1}{2} x\cdot v g''(r)d\mu_{s\ell^{-2}}(r,z) = O(t\ell^2).\end{equation}

To prove that $\calE_{t\ell^{-2}}(r)$ evolves diffusively, it remains to close the equation \eqref{diffusion energy limit 2} by replacing the kinetic energy integral $\int_{E} \frac{1}{2}|v|^2 d\mu_{s\ell^{-2}}(r,z)$, by the energy $\calE_{t\ell^{-2}}(r)$. To do so, we prove a result of equipartition of energy in Lemma \ref{equipartition Lemma}. This result states that long time integrals of kinetic energy can be actually replaced by a fraction of time integrals of the total energy. We prove it only for specific pinning potentials for which the identity
$$x\cdot \nabla U(x) = 2 U(x)$$
holds. For a positive continuously differentiable function $U$, it is well known that satisfying this identity is equivalent to be homogeneous of degree $2$. In particular, $U$ is such that $U(x) = |x|^2\psi(x/|x|)$, where $\psi\in C^2(\Sph^{d-1},\R^*_{+})$. It is straightforward to check that $U$ satisfies \ref{H1} in this setting. We furthermore require $\psi$ to be symmetric for Lemma \ref{Symmetry} to hold.

\begin{Lem}\label{equipartition Lemma}
Let $U(x) = |x|^2\psi(x/|x|)$, where $\psi\in C^2(\Sph^{d-1},\R^*_{+})$ is a symmetric function. For any $G\in\calC^2(\T)$, we have the following time equipartition of energy.
\begin{equation}\label{equipartition result} \int_0^t ds \int_{\T\times E} \frac{1}{2}|v|^2  G(r) d\mu_{s\ell^{-2}}(r,z) = \frac{1}{2} \int_0^t ds\int_\T \calE_{s\ell^{-2}}(r) G(r)dr + O((1+\bar\gamma)t\ell^2).\end{equation}
\end{Lem}

\begin{proof}
Let $(Y_s)_{s\geq 0} = (r_0,X_s,V_s)_{t\geq 0}$ be a solution to the nonlinear SDE \eqref{nonlinear mean field SDE}. We compute for any $T>0$:
\begin{align*}
&X_T\cdot V_T  - X_0\cdot V_0 \\
     &\quad = \int_0^T |V_s|^2ds - \int_0^T X_s \cdot \left(\nabla U(X_s)+X_s\right) ds + \int_0^T\int_{\R^d} X_s \cdot \left( v' - V_{s_-}\right)d\calN^{\mu,r_0}(s,v').
\end{align*}
Using that $x\cdot \nabla U(x) = 2U(x)$, we can introduce the potential $U$ at the last line. Then, using \eqref{definition energy of one particle 2}, we can introduce $\calE[\mu_s](Y_s)$ and get
\begin{align*}
&X_T\cdot V_T  - X_0\cdot V_0 \\
     &\quad = 4 \int_0^T \frac{1}{2} |V_s|^2ds  - 2 \int_{0}^{T} \calE[\mu_s](Y_s) ds + \int_{0}^{T} ds \left(\frac{1}{2}\int_{\T\times E}\Phi_\ell(r_0-r')|x'|^2d\mu_s(r',z') - \frac{1}{2}|X_s|^2\right) \\
           &\hspace{9cm}+ \int_0^T\int_{\R^d} X_s \cdot \left( v' - V_{s_-}\right)d\calN^{\mu,r_0}(s,v').
\end{align*}
Multiplying by $G(r_0)$ on both sides for $G\in\calC^2(\T)$ and taking expectations gives:
\begin{align*}
&\int_{\T\times E} x\cdot v G(r) d\mu_T(r,z) - \int_{\T\times E} x\cdot v G(r) d\mu_0(r,z) \\
    &= 4\int_0^T ds \int_{\T\times E} \frac{1}{2}|v|^2 G(r) d\mu_s(r,z)  - 2 \int_{0}^{T} ds \int_\T\calE_s(r) G(r) dr \\
    &\hspace{1cm}+ \ell^2 c_\phi \int_{0}^{T} ds \int_{\T\times E} \frac{1}{2} |x|^2 G''(r) d\mu_s(r,z) + \ell^2 \bar\gamma c_\gamma \int_{0}^{T} ds \int_{\T\times E} x\cdot v G''(r) d\mu_s(r,z) + o(T\ell^2),
\end{align*}
where we used a Taylor expansion for the two terms at the last line. Applying this result for $T=t\ell^{-2}$, and rearranging terms, we get:
\begin{align*}
&\int_0^t ds \int_{\T\times E} \frac{1}{2}|v|^2 G(r) d\mu_{s\ell^{-2}}(r,z)  - \frac{1}{2} \int_{0}^{t} ds \int_\T\calE_{s\ell^{-2}}(r) G(r) dr \\
    &\hspace{0.2cm}=  \ell^2 \frac{1}{4}\left( \int_{\T\times E} x\cdot v G(r) d\mu_{t\ell^{-2}}(r,z) - \int_{\T\times E} x\cdot v G(r) d\mu_0(r,z) \right)\\
    &\hspace{0.4cm}+ \ell^2 \frac{c_\phi}{4} \int_{0}^{t} ds \int_{\T\times E} \frac{1}{2} |x|^2 G''(r) d\mu_{s\ell^{-2}}(r,z) + \ell^2 \bar\gamma \frac{c_\gamma}{4} \int_{0}^{t} ds \int_{\T\times E} x\cdot v G''(r) d\mu_{s\ell^{-2}}(r,z) + o(t\ell^2).
\end{align*}
Applying the uniform moment bound in Lemma \ref{bound speed}, we deduce that the right-hand side is a $O((1+\bar\gamma)\ell^2)$ and this concludes the proof.
\end{proof}

Combining the equipartition result \eqref{equipartition result} and the control on the Hamiltonian current \eqref{vanishing hamiltonian current} in \eqref{diffusion energy limit 2}, we finally get that $\calE_{t\ell^{-2}}(r)$ evolves diffusively:
\begin{align}\label{diffusion energy limit 3}
&\int_{\T} \calE_{t\ell^{-2}}(r) g(r)dr - \int_{\T} \calE_0(r) g(r)dr = \bar\gamma \frac{c_\gamma}{2} \int_0^t ds \int_{\T} \calE_{s\ell^{-2}}(r) g''(r)dr + O(t\ell + t\bar\gamma\ell^2),
\end{align}
for any $g\in\calC^4(\T)$.

\subsection{Convergence of the particle system's energy}

We deduce Proposition \ref{Energy evolution} from the convergence of the microscopic energy to $\calE_t(r)$ proven in the next lemma. Let us fix $\bar\gamma$ and let $c(N,\ell,t)$ denote the constant appearing in Theorem \ref{Convergence result}, \textit{i.e}
$$c(N,\ell,t) = K_1 \left((N\epsilon_N)^{-\frac{1}{4(d+1)}} + \frac{\epsilon_N^{1/2}}{\ell^{1/2}} \right)\e^{K_2t}.$$
Recall from \eqref{definition energy particle system} the definition of the microscopic energy
$$\calE^i_t = \calE[\mu^N_t](Y^i_t) = \frac{1}{2}|V_t^i|^2 + \frac{1}{4}\int_{\T\times E} \Phi_\ell\left(\frac{i}{N}-r'\right)|X^i-x'|^2d\mu^N_t(r',z') + U(X^i_t).$$

\begin{Lem}\label{Convergence of Energy}
Let $G\in \calC^1(\T)$. Under hypothesis \ref{H9}, there exist positive constants $K,K'$ such that for any time $T>0$ and a constant $M>0$ large enough
$$\E\left[\left |\frac{1}{N}\sum_{i=1}^N \calE^i_T G\left(\frac{i}{N}\right) - \int_\T \calE_T(r) G(r) dr\right |\right] \leq M^2c(N,\ell,T) + \frac{K}{M^b}\e^{K'T} + O(\ell^2).$$
\end{Lem}

In this result, $b$ is the constant appearing in the moment hypothesis \ref{H9}. Taking $T=t\ell^{-2}$, $M= c(N,\ell,T)^{-1/3}$ and $\epsilon_N = \ell^{\frac{2d+2}{2d+3}} N^{-\frac{1}{2d+3}}$ gives the result stated in Proposition \ref{Energy evolution}. In particular, one can find a constant $c>0$ and choose $\ell = \ell(N) = c(\log N)^{-1/2}$ so that \eqref{limit with ell(N)} holds.

\begin{proof}
First, let us simplify the problem. In the definition \eqref{definition energy mu} of the energy $\calE_T(r)$, we expect that the approximation
$$\frac{1}{4} \int_{\T}\Phi_\ell(r-r') |x'|^2  d\mu_T(r',z') \approx \frac{1}{4}|x|^2$$
is true when $\ell$ is small. We can actually make this approximation rigorous when we integrate $\calE_T(r) G(r)$:
\begin{align*}
\int_{\T} \calE_T(r) G(r) dr &= \int_{\T\times E}\left(\left( \frac{1}{2}|v|^2 + \frac{1}{4} |x|^2 + U(x)\right)G(r) + \frac{1}{4}|x|^2 \int_{\T}\Phi_\ell(u)  G(r+u) du \right) d\mu_T(r,z) \\
&= \int_{\T\times E} \left( \frac{1}{2}|v|^2 + \frac{1}{2} |x|^2 + U(x) \right) G(r) d\mu_T(r,z) + O(\ell^2),
\end{align*}
by a Taylor expansion of $G$. By the same procedure for the microscopic system, we have
\begin{align*}
\frac{1}{N}\sum_{i=1}^N \calE^i_T G\left(\frac{i}{N}\right) &= \int_{\T\times E} \left( \frac{1}{2}|v|^2 + \frac{1}{2}|x|^2 + U(x) \right)G(r)d\mu^N_T(r,z) \\
     &\hspace{2cm}- \frac{1}{2}\int_{(\T\times E)^2} x\cdot x' G(r) d\mu^N_T(r,z)d\mu^N_T(r',z') + \delta_\ell,
\end{align*}
where $\delta_\ell$ is such that $\E[|\delta_\ell |] = O(\ell^2)$ by conservation of energy. Notice that we have a non-vanishing extra term for the microscopic version. We introduce
$$H(x,v) = \frac{1}{2}|v|^2 + \frac{1}{2}|x|^2 + U(x)$$
for notational convenience. The proof of the lemma now boils down to bound the following two terms
\begin{equation}\label{convergence noninteracting terms} \E\left[ \left| \int_{\T\times E} H(x,v) G(r) d\mu^N_T(r,z) -  \int_{\T\times E} H(x,v) G(r) d\mu_T(r,z) \right|\right] \end{equation}
and
\begin{equation}\label{convergence interacting terms} \E\left[ \left| \int_{(\T\times E)^2} x\cdot x' G(r) d\mu^N_T(r,z)d\mu^N_T(r',z') -  \int_{(\T\times E)^2} x\cdot x'G(r) d\mu^N_T(r,z) d\mu_T(r',z') \right|\right],\end{equation}
since the second integral in \eqref{convergence interacting terms} is null by Lemma \ref{Symmetry}.

For that, we are going to use Theorem \ref{Convergence result}. However, since the Wasserstein $\calW_1$ distance only enables to control differences of integrals with respect to Lipschitz functions with Lipschitz constant less than $1$, we have to cut the large values of $H$ and $x\cdot x'$. To control \eqref{convergence noninteracting terms}, we therefore introduce a sequence of functions $H^M$ depending on a parameter $M$ such that $H^M$ approximates $H$:
$$H^M(z) := H(z) \1_{H(z)<M} + M\left(2-\exp{\left(1-\frac{H(z)}{M}\right)} \right)\1_{H(z)\geq M}.$$
With this choice, it is easy to check that $H^M/M^2$ is a Lipschitz continuous function such that
$$\left\| \frac{1}{M^2}H^M \right \|_\infty \leq \frac{2}{M}, \hspace{1cm} Lip\left (\frac{1}{M^2}H^M \right) \leq \frac{1}{M^2}.$$
In particular, 
$$Lip\left (\frac{1}{M^2}H^M G\right) \leq Lip\left (\frac{1}{M^2}H^M \right) \|G\|_\infty + \left\| \frac{1}{M^2}H^M \right \|_\infty Lip(G) \leq \frac{1}{M^2}\|G\|_\infty +\frac{2}{M} Lip(G),$$
which is less than one for $M$ large enough. Introducing cross terms in \eqref{convergence noninteracting terms}, we now bound it by
\begin{align}\label{convergence energy 1}
&\E\left[\left | \int_{\T\times E} H(z) G(r)d\mu^N_T(r,z) - \int_{\T\times E} H(z) G(r) d\mu_T(r,z)\right |\right] \nonumber\\
    &\hspace{1.3cm} \leq \| G\|_\infty \E\left[\int_{\T\times E} \left| H(z) -H^M(z)\right | d\mu^N_T(r,z) \right] + \| G\|_\infty \int_{\T\times E} \left |  H(z) -H^M(z) \right | d\mu_T(r,z) \nonumber \\
         &\hspace{2cm} + M^2 \E\left[\left | \int_{\T\times E} \frac{1}{M^2}H^M(z) G(r)d\mu^N_T(r,z) - \int_{\T\times E} \frac{1}{M^2}H^M(z) G(r) d\mu_T(r,z)\right |\right].
\end{align}
We can bound the last term by $M^2 \E[\calW(\mu^N_t,\mu_t)]$ and get
\begin{equation} \label{convergence energy 3}
\E\left[\left | \int_{\T\times E} H^M(z) G(r)d\mu^N_T(r,z) - \int_{\T\times E} H^M(z) G(r) d\mu_T(r,z)\right |\right] \leq M^2 c(N,\ell,T).
\end{equation}
It just remains to control the first two terms in \eqref{convergence energy 1} which correspond to the cut parts. The first term in \eqref{convergence energy 1} can be bounded by
\begin{align}\label{convergence energy 2}
\E\left[\int_{\T\times E} \left| H(z) -H^M(z)\right | d\mu^N_T(r,z) \right]  &\leq \E\left[\int_{\T\times E}  H(z)\1_{H(z)\geq M} d\mu^N_T(r,z) \right] \nonumber \\
&\leq \frac{1}{M^b}\E\left[\frac{1}{N}\sum_{i=1}^N H(X^i_T,V^i_T)^{1+b}\right],
\end{align}
by a Markov inequality, with the view of using the moment hypothesis \ref{H9}. We now derive a moment type bound at time $T$. Using the dynamics \eqref{original system solid}, and the fact that $\nabla_vH = v$ and $\nabla_x H = x + \nabla U(x)$, we compute the last term:
\begin{align} \label{convergence energy 4}
&\E\left[\frac{1}{N}\sum_{i=1}^N H(X^i_T,V^i_T)^{1+b}\right] - \E\left[\frac{1}{N}\sum_{i=1}^N H(X^i_0,V^i_0)^{1+b}\right] \nonumber \\
&\hspace{3cm}= (1+b) \int_0^T \E\left[\frac{1}{N}\sum_{i=1}^N \sum_{k=-\ell N}^{\ell N} \phi_k  V^i\cdot X^{i+k} H(X^i_s,V^i_s)^{b}\right] \nonumber \\
        &\hspace{3.5cm} + \bar\gamma \int_0^T \E\left[\frac{1}{N}\sum_{i=1}^N \sum_{k=-\ell N}^{\ell N} \gamma_k \left(H(X^i_s,V_s^{i+k})^{1+b} -  H(X^i_s,V^i_s)^{1+b} \right)\right] ds,
 \end{align}
the first term being the hamiltonian contribution and the second one the stochastic contribution. We then bound the product
\begin{align*}
V^i\cdot X^{i+k} H(X^i_s,V^i_s)^{b} &\leq \left(\frac{1}{2}|V^i|^2 + \frac{1}{2}|X^{i+k}|^2 \right) H(X^i_s,V^i_s)^{b} \\
&\leq H(X^i_s,V^i_s)^{1+b} + H(X^{i+k},V^{i+k})H(X^i_s,V^i_s)^{b} \\
&\leq KH(X^i_s,V^i_s)^{1+b} + K'H(X^{i+k}_s,V^{i+k}_s)^{1+b},
\end{align*}
where $K$ and $K'$ are two constants. We bounded directly the two terms by $H$ at the second line, and then used Young's inequality at the third line. Summing this inequality over $i$ and $k$, we obtain that the first term in the right-hand side of \eqref{convergence energy 4} is thus bounded by
\begin{equation}\label{convergence energy 5}\int_0^T \E\left[\frac{1}{N}\sum_{i=1}^N \sum_{k=-\ell N}^{\ell N} \phi_k  V^i\cdot X^{i+k} H(X^i_s,V^i_s)^{b}\right]\leq K \int_0^T \E\left[\frac{1}{N}\sum_{i=1}^N H(X^i_s,V^i_s)^{1+b}\right],\end{equation}
for some constant $K$. The same bound can be obtained for the second term in the right-hand side of \eqref{convergence energy 4}. All in all, By Gronwall's inequality and \ref{H9}, we get:
$$ \E\left[\frac{1}{N}\sum_{i=1}^N H(X^i_T,V^i_T)^{1+\epsilon}\right] \leq K \e^{K'T}.$$
Eventually, the term \eqref{convergence energy 2} corresponding to the large energies is bounded by
$$\E\left[\int_{\T\times E} \left| H(z) -H^M(z)\right | d\mu^N_T(r,z) \right] \leq \frac{K}{M^b}\e^{K'T}.$$
The same bound can be obtained for the second term in \eqref{convergence energy 1}. Combining both bounds together with \eqref{convergence energy 3} in \eqref{convergence energy 1} gives the bound we wanted for the convergence of $H$ in \eqref{convergence noninteracting terms}:
$$\E\left[\left | \int_{\T\times E} H(z) G(r)d\mu^N_T(r,z) - \int_{\T\times E} H(z) G(r) d\mu_T(r,z)\right |\right] \leq M^2c(N,\ell,T) + \frac{K}{M^b}\e^{K'T}.$$
\eqref{convergence interacting terms} can be controlled in the same spirit, and this concludes the proof.

\end{proof}


\appendix


\section{Convergence of some empirical measure for the Wasserstein distance}\label{Appendix A}

For any $1\leq j \leq \epsilon_N^{-1}$, let $(\bar Y^i)_{i\in NB_j}$ be a family of independent random variables on $B_j\times\R^d\times\R^d$ with law $\mu^j(r,z) = \epsilon_N^{-1}\mu(r,z)\1_{r\in B_j}$, and $\bar Y^i = (\bar r^i, \bar X^i, \bar V^i)$. In this section, we prove that the empirical measure $\bar \mu^{N,j} = 1/(N\epsilon_N)\sum\delta_{\bar Y_i}$ satisfies a law of large numbers in some sense. More precisely, we prove that

\begin{Prop}
If $\mu$ has a finite moment of order $2$ in the sense that $\int_{\T\times E} |z|^2 d\mu(r,z) <\infty$. Then there exists some positive constant $K$ such that
$$\epsilon_N \sum_{j=1}^{\epsilon_N^{-1}} \E\left[\calW_1(\bar \mu^{N,j}, \mu^j)\right] \leq K (N\epsilon_N)^{-\frac{1}{4(d+1)}} + K\epsilon_N.$$
\end{Prop}

Notice that more refined versions of law of large numbers for empirical measures in the Wasserstein distance exist (see \cite{Fournier} for instance). But the proposition we prove here has an easy proof and is well suited for our problem.

\begin{proof}

Let $M>0$. For any $1\leq j \leq \epsilon_N^{-1}$, using Kantorovich-Rubinstein duality formula, we can bound:
\begin{align}\label{term appendix}
\E\left[\calW_1(\bar \mu^{N,j}, \mu^j)\right] &\leq \E\left[\sup_{\substack{Lip(\varphi)\leq 1 \\\varphi(0)=0}} \int_{B_j\times E}\varphi(r,z)\1_{\{| z | \leq M\}} d(\bar\mu^{N,j} - \mu^j) \right] \nonumber \\
   &\hspace{3cm}+ \E\left[\sup_{\substack{Lip(\varphi)\leq 1 \\\varphi(0)=0}} \int_{B_j\times E}\varphi(r,z)\1_{\{| z | > M\}} d(\bar\mu^{N,j} - \mu^j) \right].
\end{align}
The second term in \eqref{term appendix} can be bounded using Markov inequality type arguments. First, we bound it by
\begin{align*}
&\E\Bigg[\sup_{\substack{Lip(\varphi)\leq 1 \\\varphi(0)=0}} \int_{B_j\times E}\varphi(r,z)\1_{\{| z | > M\}} d(\bar\mu^{N,j} - \mu^j) \Bigg] \\
    &\hspace{4cm}\leq \E\left[\int_{B_j\times E} \left(r + |z|\right)\1_{\{| z | > M\}} d(\bar \mu^{N,j} + \mu^j)\right] \nonumber \\
    &\hspace{4cm}\leq 2 \epsilon_N^{-1} \int_{B_j\times E}  \1_{\{| z | > M\}} d\mu\left(r,z\right) +2 \epsilon_N^{-1}\int_{B_j\times E} |z| \1_{\{| z | > M\}}d\mu(r,z).
\end{align*}
Taking the mean over $j$ and then applying Markov inequality in the last expression gives:
\begin{align} \label{first term appendix}
&\epsilon_N \sum_{j=1}^{\epsilon_N^{-1}} \E\Bigg[\sup_{\substack{Lip(\varphi)\leq 1 \\\varphi(0)=0}} \int_{\T\times E}\varphi(r,z)\1_{\{| z | > M\}} d(\bar\mu^{N,j} - \mu^j) \Bigg] \nonumber\\
&\hspace{6cm}\leq 2 \int_{\T\times E}  \1_{\{| z | > M\}} d\mu\left(r,z\right) + 2 \int_{\T\times E} |z| \1_{\{| z | > M\}}d\mu(r,z)\nonumber \\
&\hspace{6cm}\leq \frac{K}{M},
\end{align}
for some constant $K$. To bound the first term in \eqref{term appendix}, let us subdivide $[-M,M]^{2d}$ in $n^{2d}$ disjoint cubes $(C_k)_{1\leq k \leq n^{2d}}$, for some integer $n\geq 1$. More precisely, each cube $C_k$ is of the form $[i_1M/n, (i_1+1)M/n]\times...\times[i_{2d}M/n, (i_{2d}+1)M/n]$ for $-n\leq i_1,...,i_{2d}\leq n-1$. Let $z_k\in C_k$ be a point in the cube $C_k$, let us say the center of $C_k$ to fix ideas. Then for any $\varphi$  such that $Lip(\varphi) \leq 1$ and $\varphi(0)=0$, one has
\begin{align}\label{second term appendix}
&\left | \int_{B_j\times E}\varphi(r,z)\1_{\{| z | \leq M\}} d(\bar\mu^{N,j} - \mu^j) \right | \nonumber \\
&\hspace{0.5cm}= \left | \sum_{k=1}^{n^{2d}} \int_{B_j\times C_k} \varphi(r,z) d(\bar\mu^{N,j} - \mu^j)\right | \nonumber\\
&\hspace{0.5cm}\leq \sum_{k=1}^{n^{2d}}\left( \left | \int_{B_j\times C_k}\varphi (j\epsilon_N,z_k) d(\bar\mu^{N,j} - \mu^j) \right | + \int_{B_j\times C_k} \left | \varphi(r,z) - \varphi(j\epsilon_N,z_k) \right | d(\bar\mu^{N,j}+\mu^j) \right) \nonumber\\
&\hspace{0.5cm}\leq \sum_{k=1}^{n^{2d}} \left( \int_{B_j\times C_k}\left( 1 + |z_k| \right)d\left|\bar\mu^{N,j} - \mu^j \right |  + \left(\epsilon_N + \frac{M}{n}\right)\left(\bar\mu^{N,j}(B_j\times C_k) + \mu^j(B_j\times C_k)\right) \right) \nonumber\\
&\hspace{0.5cm}\leq (1+M) \sum_{k=1}^{n^{2d}} \left | \bar\mu^{N,j}(B_j\times C_k) - \mu^j(B_j\times C_k) \right |  + 2\frac{M}{n} + 2 \epsilon_N,
\end{align}
where we used at the third line that $\varphi(j\epsilon_N,z_k) \leq r + |z_k|$ for the first term, and the Lipschitz property for the second term. Then we used that $|z_k| \leq M$ at the last line. For any $k\leq n^{2d}$ and any $1\leq j \leq \epsilon_N^{-1}$, by independence of the $(\bar Y^i)_{i\in NB_j}$ we get
\begin{align} \label{third term appendix}
\E \left[ \left ( \bar\mu^{N,j}(B_j\times C_k) - \mu^j(B_j\times C_k) \right )^2 \right] &= \frac{1}{(N\epsilon_N)^2}\sum_{i,i'\in NB_j} \E \left[  \1_{\bar Z^i\in C_k} \1_{\bar Z^{i'}\in C_k} \right] \nonumber \\
&\hspace{0.5cm}- \frac{2}{N\epsilon_N}\sum_{i\in NB_j} \Prb(\bar Z^i\in C_k)\mu^j(B_j\times C_k) + \mu^j(B_j\times C_k)^2 \nonumber \\
&= \frac{1}{N\epsilon_N} \mu^j(B_j\times C_k) \left(1 - \mu^j(B_j\times C_k)\right) \nonumber \\
&\leq \frac{1}{N\epsilon_N} \mu^j(B_j\times C_k).
\end{align}
Taking the mean over $j$ in \eqref{second term appendix} and using \eqref{third term appendix} gives, by Cauchy-Schwarz inequality
\begin{align*}
&\epsilon_N \sum_{j=1}^{\epsilon_N^{-1}} \E\left[\sup_{\substack{Lip(\varphi)\leq 1 \\\varphi(0)=0}} \int_{B_j\times E}\varphi(r,z)\1_{\{| z | \leq M\}} d(\bar\mu^{N,j} - \mu^j) \right] \\
&\hspace{3cm} \leq (1+M) \epsilon_N \sum_{j=1}^{\epsilon_N^{-1}} \sum_{k=1}^{n^{2d}} \E\left[\left | \bar\mu^{N,j}(B_j\times C_k) - \mu^j(B_j\times C_k) \right |\right]  + 2\frac{M}{n} + 2\epsilon_N \\
&\hspace{3cm} \leq \frac{ 2M \epsilon_N }{(N\epsilon_N)^{1/2}} \sum_{j=1}^{\epsilon_N^{-1}} \sum_{k=1}^{n^{2d}} \sqrt{\mu^j(B_j\times C_k)} + 2\frac{M}{n}+2\epsilon_N,
\end{align*}
for $M>1$. Applying Cauchy-Schwarz inequality to the two sums at the last line and using that for any $1\leq j \leq \epsilon_N^{-1}$, we have $\sum_{k=1}^{n^{2d}} \mu^j(B_j\times C_k) \leq 1$, we get
\begin{equation}\label{fourth term appendix}
\epsilon_N \sum_{j=1}^{\epsilon_N^{-1}} \E\left[\sup_{\substack{Lip(\varphi)\leq 1 \\\varphi(0)=0}} \int_{B_j\times E}\varphi(r,z)\1_{\{| z | \leq M\}} d(\bar\mu^{N,j} - \mu^j) \right] \leq \frac{ 2M n^d }{\left(N\epsilon_N\right)^{1/2}} + 2\frac{M}{n} + 2\epsilon_N.
\end{equation}
Combining \eqref{first term appendix} and \eqref{fourth term appendix} into \eqref{term appendix} finally gives
$$\epsilon_N \sum_{j=1}^{\epsilon_N^{-1}} \E\left[\calW_1(\bar \mu^{N,j}, \mu^j)\right]  \leq K\left(\frac{Mn^{d}}{(N\epsilon_N)^{1/2}} + \frac{M}{n} + \epsilon_N + \frac{1}{M}\right),$$
which choosing $M=(N\epsilon_N)^{\frac{1}{4(d+1)}}$ and $n=\lfloor M^2 \rfloor$ yields the desired result.

\end{proof}

\section*{Acknowledgements}
The author would like to thank Thierry Bodineau for suggesting this subject, and for the useful discussions and suggestions. The author would also like to thank Joaqu\'in Fontbona and Nicolas Fournier for their relevant remarks.


\bibliographystyle{plain}
\bibliography{biblio}

\end{document}